\documentclass[11pt]{amsart}
\usepackage{amssymb}

\usepackage{palatino}
\input amssym.def

\usepackage{amsmath, amsfonts}
\usepackage{amssymb}
\usepackage{amscd}

\usepackage[mathscr]{eucal}
\usepackage{palatino}
\newfont{\cyrr}{wncyr10}

\newcommand{\N}{{\mathbb N}}
\newcommand{\Z}{{\mathbb Z}}

\newcommand{\R}{{\mathbb R}}
\newcommand{\C}{{\mathbb C}}

\renewcommand{\mod}{{\, \rm mod \, }}

\newtheorem{thm}{Theorem}

\newtheorem{lem}[thm]{Lemma}
\newtheorem{cor}[thm]{Corollary}
\newtheorem{prop}[thm]{Proposition}

\begin{document}

\title[Sign changes]{Sign changes of Fourier coefficients
of Siegel cusp forms of degree two on 
Hecke congruence subgroups}

\author{S. Gun and J. Sengupta}

\address[S. Gun]   
{Institute of Mathematical Sciences, 
HBNI,
C.I.T Campus, Taramani, 
Chennai  600 113, 
India.}
\email{sanoli@imsc.res.in}

\address[J. Sengupta]   
{School of Mathematics,
Tata institute of Fundamental research, 
Homi Bhabha Road, Mumbai 400 005, 
India.}
\email{sengupta@math.tifr.res.in}



\maketitle

\begin{abstract} 
In this article, we give a lower bound on the number
of sign changes of Fourier coefficients of a non-zero
degree two Siegel cusp form of even integral weight
on a Hecke congruence subgroup. We
also provide an explicit upper
bound for the first sign change of Fourier coefficients
of such Siegel cusp forms. Explicit upper
bound on the first sign change of Fourier
coefficients of a non-zero Siegel cusp form of even
integral weight on
the Siegel modular group for arbitrary genus
were dealt in an earlier work of 
Choie, the first author and Kohnen.
\end{abstract}

\smallskip

\section{\large Introduction and Statements of the Theorems}

\smallskip

The arithmetic of Fourier coefficients of cusp forms 
has been the focus of study for sometime now. When these 
coefficients are real, studying distribution
of their sign has become an active area of research in recent times.
For example, in the case of elliptic cusp forms with real coefficients,
this has been studied in  \cite{{CK},{GS},{IKS},{KKP},{KM}, {RM}}.

The question of infinitely many sign changes of Fourier coefficients 
of Siegel cusp forms with real Fourier coefficients has been studied 
in \cite{SJ} and sign changes
of Hecke eigen values of Siegel cusp forms of degree two was studied
in \cite{WK}.  The first sign change question for Hecke eigen values
for Siegel cusp forms of genus two was addressed by Kohnen
and the second author \cite{KS1}.

In this article, we give a lower bound on the number of
sign changes in short intervals of Fourier coefficients 
of non-zero Siegel cusp forms of even integral 
weight and degree $2$ on the 
Hecke congruence subgroup $\Gamma_0^{(2)}(N)$
with real Fourier coefficients.
In order to deduce this result, we need to prove a 
corresponding result for elliptic cusp forms of square 
free level. This result seems to be new even in the case 
of elliptic cusp forms (see Theorem \ref{thmthree})
and requires us to redo an earlier work
of Rankin \cite{RR} with explicit 
dependence on the weight and level of the elliptic
cusp form (see Proposition \ref{secle}).
If the elliptic cusp form is a normalised new form
of square free level $N$, then Kohnen, Lau and 
Shparlinski \cite{KLS} give a better lower bound. 

On another direction, recently Choie, the first author and Kohnen
\cite{CGK} gave an explicit upper bound for the first sign change of 
Fourier coefficients of non-zero Siegel cusp forms of even integral 
weight on the symplectic group
$\Gamma_g:= Sp_g(\Z) \subset GL_{2g}(\Z)$ of arbitrary genus
$g \ge 2$ with real Fourier coefficients. In this article,
we provide an explicit upper bound
for the first sign change of Fourier coefficients of
degree two non-zero Siegel cusp form of
even integral weight on a Hecke congruence 
subgroup of $Sp_2(\Z)$ with
real Fourier coefficients.

\medskip

In order to state our theorems, we now fix some
notations. For a natural number $N$, let 
$$
\Gamma_{0}^{(2)}(N) := \{ M= \begin{pmatrix}
                                      A & B \\
                                      C & D
                                      \end{pmatrix}
                                \in Sp_2(\Z)  ~~|~~  C \equiv 0 \!\!\!\!\! \pmod{N} \}
$$ 
be the Hecke congruence subgroup
of $\Gamma_2$ of level $N$ and $\mathcal{H}_2$ be
the Siegel upper-half space 
of degree $2$ consisting of all
symmetric $2 \times 2$ complex matrices whose
imaginary parts are positive definite. For $k \in \N$,
let $S_k(\Gamma_0^{(2)}(N))$ be the space
of Siegel cusp forms of weight $k$
on $\Gamma_0^{(2)}(N)$.  It is well known
that
$$
\Psi_1(N) ~:=~  [ \Gamma_2 : \Gamma_0^{(2)}(N)] 
~=~ 
N^3 \prod_{ p |N, \atop p \text{ prime }} (1+1/p)(1+ 1/p^2).
$$
Any $F \in S_k(\Gamma_0^{(2)}(N))$ 
has a Fourier expansion 
\begin{equation}\label{eq1}
F(Z):= \sum_{T > 0} a(T) ~e^{2\pi i \text{ tr }(TZ)},
\end{equation}
where $Z \in \mathcal{H}_2$ and
$T$ runs over all positive definite half-integral
$2 \times 2$ matrices. In this set-up, we have
the following theorems.

\begin{thm}\label{thmtwo}
Let $k, N$ be natural numbers with $k$ even 
and $N$ square-free. Also let 
$F$ be a non-zero Siegel 
cusp form of weight $k$ and degree $2$ on the 
Hecke congruence subgroup $\Gamma_0^{(2)}(N)$
with real Fourier coefficients $a(T)$ for $T> 0$. 
For any $\epsilon > 0$, let 
$h := x^{13/14 ~+~ \epsilon}$. Then
there exists $T_1 > 0, T_2 >0$
with $\text{ tr}~T_i \in (x, x + h], ~ i=1, 2$
such that $a(T_1)> 0$ and $a(T_2)< 0$
for any
\begin{eqnarray*}
x 
&\gg_{\epsilon}& 
k^{42}.~N^{84}.~\log^{80}k.
~e^{(c_7 \frac{\log(N+1)}{\log\log(N+2)})},
\end{eqnarray*}
where $c_7 > 0$ is an absolute constant
and the constant in $\gg$ depends only on $\epsilon$. 
\end{thm}

As an immediate corollary, we have

\begin{cor}\label{corone}
Let $k, N$ be natural numbers with $k$ even 
and $N$ square-free. Also let 
$F$ be a non-zero Siegel 
cusp form of weight $k$ and degree $2$ on the 
Hecke congruence subgroup $\Gamma_0^{(2)}(N)$
with real Fourier coefficients $a(T)$ for $T> 0$. 
Then for any sufficiently small $\epsilon > 0$ and 
\begin{eqnarray*}
x 
&\gg_{\epsilon}& 
k^{42}.~N^{84}.~\log^{80}k.~e^{(c_7 \frac{\log(N+1)}{\log\log(N+2)})}
~~\text{ with }~~~ c_7 > 0 ~\text{ an absolute constant},
\end{eqnarray*}
there exists at least $\gg_{\epsilon} x^{1/14 - \epsilon}$ many 
$T> 0$ with $\text{ tr}~T \in (x, 2x]$ and
$a(T) < 0$. The same lower bound holds
for the number of $T > 0$ with $\text{ tr}~T \in (x, 2x]$
and $a(T)> 0$. 
\end{cor}

\begin{thm}\label{thmone}
Let $k, N$ be natural numbers with $k$ even 
and $N$ square-free. Also let 
$F$ be a non-zero Siegel 
cusp form of weight $k$ and degree $2$ on the 
Hecke congruence subgroup $\Gamma_0^{(2)}(N)$
with real Fourier coefficients $a(T)$ for $T> 0$
at infinity. Then there exists $T_1 > 0, T_2 > 0$ with 
$$
\text{  tr }T_1, ~~ \text{  tr }T_2 \ll  
 ~k^5 (\log k)^{26} ~N^{39/2}   e^{c_2 
 \frac{\log(N+1)}{\log\log(N+2)}}
$$
and $a(T_1) > 0,  a(T_2) < 0$.
Here $c_2$ as well as the constant in $\ll$ are absolute.
\end{thm}

The paper is organized as follows. In the next section,
we give a proof of Theorem \ref{thmone} since it is 
relatively easier. In the penultimate section, 
using strong convexity principle, 
we give a lower bound on the number of sign changes in 
short intervals of Fourier coefficients of elliptic cusp forms 
of square-free level with real Fourier coefficients. 
In the final section, we
give a proof of Theorem \ref{thmtwo}.

\smallskip

\section{\large Proof of Theorem \ref{thmone}}

\smallskip

By the given hypothesis, we have
$$
F(Z):= \sum_{T > 0} a(T) ~e^{2\pi i \text{ tr }(TZ)},
$$
where $Z \in \mathcal{H}_2$ and $T$ runs over positive 
definite half-integral $2 \times 2$ matrices.  
Write 
$$
T = \begin{pmatrix}  n & r/2  \\  r/2 & m \\ \end{pmatrix}
$$
with $m,n \in \N$ and $ r \in \Z$ with $r ^2 < 4nm$
and 
$$
Z := \begin{pmatrix} \tau & z \\ z &  \tau' \\  \end{pmatrix},
$$
where $\tau, \tau' \in \mathcal{H}, z \in \C$
with $\Im \tau\Im\tau' - \Im^2 z > 0$.
Then 
$$
F(Z) := \sum_{m \ge 1}  \phi_m (\tau, z) e^{2 \pi i m \tau' },
$$
where
\begin{equation*}
\phi_m( \tau, z) 
= 
\sum_{n \ge 1, r \in \Z, \atop r^2 < 4nm} c(n,r) 
~e^{2 \pi i (n \tau + r z)} 
\phantom{m}\text{ and }\phantom{m}
c(n,r) := a\left(\begin{pmatrix}  n & r/2  \\  r/2 & m \\ \end{pmatrix}\right).
\end{equation*}
We claim that $\phi_m$ is a Jacobi cusp form of 
weight $k$ and index $m$ on $\Gamma_0(N) \ltimes \Z^2$.

For  $M:= \begin{pmatrix}  a & b \\  c & d \\ \end{pmatrix}  \in SL_2(\Z)$
and $(\lambda, \mu) \in \Z^2$, the matrices
$$
\gamma(M): = \begin{pmatrix} a & 0 & b & 0 \\ 0 & 1 & 0 & 0 \\ 
c &  0 & d & 0 \\  0 & 0 &  0 & 1\\ \end{pmatrix} \in Sp_2(\Z)
\phantom{m} \text{and} \phantom{m} 
M_{\lambda, \mu}:= \begin{pmatrix} 1 & 0 & 0 &  \mu \\ \lambda & 1 & \mu & 0 \\
0 & 0 & 1 & - \lambda \\ 0 & 0 & 0 & 1 \\ \end{pmatrix}
\in \Gamma_0^{(2)}(N).
$$ 
Note that if $M \in \Gamma_0(N)$, then $\gamma(M)
\in \Gamma_0^{(2)}(N)$.
These matrices $\gamma(M)$ and $M_{\lambda, \mu}$ 
act on $\mathcal{H}_2$ as follows 
\begin{eqnarray*}
( \tau, z, \tau') & \mapsto &  \left(\frac{a\tau + b}{c\tau + d}, ~
\frac{z}{c\tau + d}, ~\tau' - \frac{cz^2}{c\tau + d} \right), \\
(\tau, z, \tau') & \mapsto & ( \tau, ~ z + \lambda\tau +
 \mu, ~\tau' + 2\lambda z + \lambda^2 \tau).
\end{eqnarray*}
Then
\begin{eqnarray}\label{eq1}
F | ~\gamma(M) (Z) 
& = &  
\sum_{m \ge 1}  ( \phi_m  ~|~ M ) (\tau, z) ~e^{2 \pi i m \tau' }  \\
\phantom{m} \text{and} \phantom{mm}
F | ~M_{\lambda, \mu} (Z) 
&=&  
\sum_{m \ge 1}  ( \phi_m  ~|~ [\lambda, \mu ] ) (\tau, z) ~e^{2 \pi i m \tau' }  
~=~
\sum_{m \ge 1} \phi_m  (\tau, z) ~e^{2 \pi i m \tau' },  \nonumber
\end{eqnarray}
as $F | ~M_{\lambda, \mu} (Z) = F(Z)$.
Also
\begin{eqnarray*}
F | \gamma(M) (Z) 
~=~  
\sum_{T > 0} b (T, \gamma(M)) ~e^{\frac{2\pi i}{N} \text{ tr }(TZ)},
\end{eqnarray*}
where $T$ runs over positive definite half-integral
$2 \times 2$ matrices.  Write
$T :=  \begin{pmatrix} n & r/2  \\  r/2 &  m_1 \\  \end{pmatrix}$,
where $n, m_1  > 0$ and $r^2 < 4 nm_1$ with $r \in \Z$
and as before
$Z := \begin{pmatrix} \tau & z \\ z &  \tau' \\  \end{pmatrix}$
with $\tau, \tau' \in \mathcal{H}, z \in \C$ and
$\Im \tau\Im\tau' - \Im^2 z > 0$. 
Then
\begin{eqnarray*} 
F | \gamma(M) (Z) 
&=&  
\sum_{T > 0} b (T, \gamma(M)) ~e^{\frac{2\pi i}{N} \text{ tr }(TZ)} \\
&=&
\sum_{n, m_1 > 0 \atop
{r \in \Z, \atop r^2 < 4n m_1}} b \left(\begin{pmatrix} n & r/2  
\\  r/2 &  m_1 \\  \end{pmatrix},
 ~\gamma(M)\right) ~e^{\frac{2\pi i}{N} (n \tau +  r z +  m_1 \tau')}.
\end{eqnarray*}
Comparing this with equation \eqref{eq1}, we see that $m_1 \equiv 0 \mod N$ 
and hence
$$
( \phi_m ~| ~ M ) (\tau, z) 
=
\sum_{n \ge 1, r \in \Z \atop  r^2 <  4mnN } 
b \left(\begin{pmatrix} n & r/2  \\  r/2 &  N m \\  \end{pmatrix},
 \gamma(M)\right) ~e^{\frac{2\pi i}{N} (n \tau +  r z)}.
$$
Hence
$$
( \phi_m ~| ~ M) (\tau, z) 
=
\sum_{n \ge 1, r \in \Z, \atop r^2  <  4mnN} 
b(n,  r ) ~e^{\frac{2\pi i}{N} (n \tau +  r z)},
$$
where 
$$
b(n,r) := b\left( \begin{pmatrix} n & r/2  \\  r/2 &  N m \\  \end{pmatrix},  
 \gamma(M) \right).
$$
Further for $M \in \Gamma_0(N)$, one has
$$
\sum_{m \ge 1}  ( \phi_m  ~|~ M ) (\tau, z) ~e^{2 \pi i m \tau' }
~= ~  
F | ~\gamma(M) (Z) 
~=~ 
F (Z)  
~=~ 
\sum_{m \ge 1}  \phi_m (\tau, z) ~e^{2 \pi i m \tau' } . 
$$
Thus $\phi_m$ is Jacobi cusp form of weight $k$
and index $m$ on $\Gamma_0(N) \ltimes \Z^2$.

Since we can replace $F$ by $-F$, it
is sufficient to show that there exists a $T> 0$ with 
$\text{ tr }T$ in the given range such that $a(T) < 0$.
We know from \cite{{PY}, {PY1}} that
there exists a $T_0 > 0$ with 
\begin{equation}\label{bound}
\text{ tr } T_0 
\le
\frac{4}{3\sqrt{3}\pi}~k \Psi_1(N) 
\end{equation}
such that $a(T_0) \ne 0$.
Write 
$T_0 := \begin{pmatrix} n_0  &  r_0/2 \\  
r_0/2 &  m_0 \\ \end{pmatrix} > 0$.
Since $a(T_0) \ne 0$, we have $\phi_{m_0}(\tau, z)$ 
is not the zero function.  Further,
$k$ is even and $- I \in \Gamma_0(N)$
implies that $\phi_{m_0}(\tau, z)$ is an even function
of $z$ for fixed $\tau$.
Suppose that
\begin{eqnarray}\label{eq2}
\phi_{m_0}( \tau, z) = \sum_{n \ge 1, r \in \Z, 
\atop r^2 < 4nm_0} c(n,r) 
~q^{n} \zeta^{r}, 
\phantom{m}\text{where} \phantom{m}
q : = e^{2 \pi i \tau}, ~ \zeta := e^{2 \pi i z}.
\end{eqnarray}
be the expansion of $\phi_{m_0}$ at infinity and
\begin{eqnarray}\label{eq3}
\phi_{m_0} (\tau, z) 
&:=& 
\sum_{\nu \ge 0} \chi_{\nu} (\tau) z^{\nu}
\end{eqnarray}
be the Taylor series expansion of $\phi_{m_0}$ 
around $z=0$. Since $\phi_{m_0}$ is not the 
zero function, all $\chi_{\nu}$'s can not be zero.
If $\alpha$ is the smallest non-negative integer
such that $\chi_{\alpha}(\tau)$ is not the zero
function, then it follows from 
\cite{EZ}, page 31 (see also
\cite{CGK}) that
$\chi_{\alpha}(\tau)$ is a non-zero cusp form
of weight $k + \alpha$ on $\Gamma_0(N)$.
Moreover, by Theorem 1.2, page 10 in \cite{EZ},
it follows that 
$$
\alpha \le 2m_0 < 2 \text{ tr } T_0  \leq
 \frac{8}{3 \sqrt{3} \pi} k \Psi_1(N).
$$
Note that $\alpha$ is even as $\phi_{m_0}$
is an even function of $z$ for fixed $\tau$.
Now by differentiating both sides of equation ~\eqref{eq3} 
with respect to $z$ and then evaluating
at $z=0$ and using equation \eqref{eq2}, we see
that for $\tau \in \mathcal{H}$,
$$
\chi_{\alpha} (\tau)
~=~   
\sum_{ n\ge 1} B(n) q^n,
$$ 
where 
\begin{equation}\label{eq4}
B(n) := \left\{ \begin{array}{ll}
\sum_{r \in \Z \atop r^2 <  4nm_0} 
c(n,  r) & {\rm if} ~ \alpha = 0, \\
\frac{1}{\alpha !} (2 \pi i)^{\alpha}
\sum_{r \in \Z \atop r^2 <  4nm_0} 
c(n,  r) r^{\alpha} & {\rm if} ~ \alpha > 0.
\end{array} \right.  
\end{equation}
In the above situation, either $i^{\alpha} =1$ 
or $i^{\alpha} =-1$. If $i^{\alpha} =1$, then by
the work of Choie and Kohnen \cite{CK}, there exists
$n_1 \ll \Psi_2(k + \alpha,~ N)$ such that $B(n_1) < 0$,
where
$$
\Psi_2(k, N) :=  k^3 N^4 \log^{10}(kN) 
~e^{c_1\frac{\log(N+1)}{\log\log(N+2)}}
~\text{ max }
\left( \prod_{p|N \atop p \text{ prime }} \frac{\log(kN)}{\log p},~  
k^2 N^{1/2} \log^{16}(kN) \right)
$$
and $c_1$ is an absolute constant.
In the case $i^{\alpha} =-1$, again by the the same result of
 Choie and Kohnen, we can find $n_2 \ll 
 \Psi_2(k + \alpha,~ N)$ such that $B(n_2) > 0$.
In both the cases using \eqref{eq4}, 
there exists an $r_1 ({\rm resp. } ~r_2)$ such that
$c(n_1, r_1) < 0  ~({\rm resp.} ~c( n_2, r_2) < 0)$. 
Thus
$$
c(n_i, r_i) 
=
a( \begin{pmatrix} n_i &  r_i/2  \\   r_i/2 &  m_0 \end{pmatrix})
< 0,
\phantom{m} \text{ where  } i =1, 2.
$$
Now for $i= 1, 2$, we have
\begin{eqnarray*}
\text{ tr } (\begin{pmatrix} n_i &  r_i/2  \\   r_i/2 &  m_0 \end{pmatrix})
&=&  n_i + m_0
 \ll   \Psi_2(k + \alpha,~ N)  + \text{ tr } T_0 \\
&\le & 
\Psi_2(k + \alpha, ~N)
+ \frac{4}{3 \sqrt{3} \pi} k \Psi_1(N) \ll \Psi_2(k + \alpha, N),
\end{eqnarray*}
where
$$ 
k + \alpha \leq k + 2m_0 < k + \frac{8k}{3\sqrt{3}\pi} \Psi_1(N) 
< 2 k \Psi_1(N).
$$
Hence
$$
\text{ tr } (\begin{pmatrix} n_i &  r_i/2  \\   r_i/2 &  m_0 \end{pmatrix})
\ll \Psi_2( 2k\Psi_1(N), N) 
\ll \Psi_2( k\Psi_1(N), N).
$$
Further, for $\epsilon > 0$
$$
\nu(N) \le (1+ \epsilon) \frac{\log N}{\log\log N}
$$
for all $N \ge N(\epsilon)$ (see \cite{GT}, page 83
for a proof).   Now $\ell := k \Psi_1(N) > N^3$
and $\frac{\log x}{\log\log x}$ is an increasing 
function for $x \ge 16$. Thus  for all $N \ge N_0( \epsilon)\ge 16$,
$$
\prod_{ p | N  \atop p \text{ prime }} \frac{\log (\ell N)}{\log p}
~ \le ~ \frac{1}{\log 2}(\log (\ell N))^{\nu(N)}
 <  ~\frac{1}{2}~e^{(1+\epsilon) \frac{\log \ell}{\log \log \ell}(\log \log (\ell N))}
 ~\le~ 
 \frac{1}{2}~e^{(1+\epsilon)^2 \log  \ell}
$$
and hence
\begin{equation}\label{new}
\Phi_{\ell}(N) := 
\prod_{ p | N  \atop p \text{ prime }} \frac{\log (\ell N)}{\log p}
~ \ll~ \ell^2 N^{1/2} \log^{16}(\ell N),
\end{equation}
where the constant in $\ll$ is absolute.
This implies that
\begin{eqnarray*}
\Psi_2(k\Psi_1(N) , ~N) 
&\ll&
e^{c_1 \frac{\log(N+1)}{\log\log(N+2)}} 
 ~k^5 N^{39/2} \log^{26}(k N^4) \log\log^{10}(N+2) \\
&\ll&
 ~k^5  (\log k)^{26} N^{39/2}
 ~e^{c_2 \frac{\log(N+1)}{\log\log(N+2)}}, 
\end{eqnarray*}
where $c_2 > 0$ is an absolute constant. The last
inequality is true as for $ k \ge 2$ and $N \ge 1$, 
we see that
$$
\log\log(N+2) 
~\ll~
e^{\frac{\log(N+1)}{\log\log(N+2)}} 
\phantom{m}\text{and}\phantom{m}
\log (kN^4)
~\ll~
(\log k) ~e^{\frac{\log(N+1)}{\log\log(N+2)}} 
$$
This completes the proof of Theorem \ref{thmone}.

\medskip

\section{\large Number of sign changes of Fourier coefficients
in short intervals for elliptic cusp forms}

\smallskip

Throughout this section, we denote a prime number by $p$.
In order to prove Theorem \ref{thmtwo},  we will need to
prove a corresponding theorem for arbitrary non-zero 
elliptic cusp form of weight $k$ on $\Gamma_0(N)$
having real Fourier coefficients.
In particular, we prove the following theorem.
\begin{thm}\label{thmthree}
Let $f$ be a non-zero cusp form of even weight $k$ 
on $\Gamma_0(N)$,
where $N$ is square free. Assume that the Fourier coefficients
$\beta(n)$ of $f$ at infinity are real. 
Then for any $\epsilon > 0$ and for 
$$
x 
~~\gg_{\epsilon}~~ 
k^{42}.~N^{84}.~ 
\log^{80}k.~ e^{c_7 ~\frac{\log(N+1)}{\log\log(N+2)}}
~~\text{ with }~~~ c_7> 0~\text{ an absolute constant},
$$
there exists $n_1, n_2 \in (x, x +h]$ with $h := x^{13/14 + \epsilon}$
such that $\beta(n_1) > 0, \beta(n_2) < 0$. The constant in
$\gg$ depends only on $\epsilon$.
\end{thm}
In order to prove this theorem, we use strong convexity
principle. Further, we need the following additional notations 
and lemmas from the works of \cite{{CK},{KS}}.

Let $\mathcal{H}$ be the Poincar\'e upper half plane, 
$S_k(N)$ be the space of cusp forms of weight
$k$ for $\Gamma_0(N)$ and $S_k^{\text{ new}}(N)$
be the subspace of new forms of $S_k(N)$.
If $f, g \in S_k(N)$, then we normalize the Petersson inner product 
on $S_k(N)$ by
$$
<f, g> := \frac{1}{[\Gamma_1 : \Gamma_0(N)]}
\int_{\Gamma_0(N) \setminus \mathcal{H}} f(z)\overline{g(z)} 
y^k \frac{dx dy}{y^2},
$$ 
where $z := x + iy$. Further, if
$$
f(z):= \sum_{n\ge} a_f(n) q^n  \in S_k(N), 
\phantom{m} q := e^{2\pi i z}, ~~~z \in \mathcal{H},
$$  
then we denote by 
$$
\lambda_f(n) : = \frac{a_f(n)}{n^{(k-1)/2}},
$$
the normalized Fourier coefficient of $f$.
By a normalized Hecke eigen form $f \in S_k^{\text{ new }}(N)$,
we mean $\lambda_f(1) =1$ and $\lambda_f(n)$ are normalized
Hecke eigenvalues. Hence by Deligne's bound, one has 
$$
|\lambda_f(n)| \le \tau(n),  \phantom{m}\text{for all }n\in \N
$$
when $f$  is a normalized Hecke eigen form in 
$S_k^{\text{ new }}(N)$ and $\tau(n)$ is the number of
divisors of $n$.
For $f, g \in S_k(N)$ and $s \in \C$ with $\sigma:= \Re(s)>1$, 
one defines the Rankin-Selberg zeta function attached to $f,g$
as
$$
R_{f,g}(s) := \sum_{ n\ge 1} \frac{\lambda_f(n){\overline
{\lambda_g(n)} } }{n^s}.
$$ 
It is known that the completed Rankin-Selberg zeta function
$$
R_{f,g}^{*}(s) := N^s (2\pi)^{-2s} \Gamma(s)\Gamma(s+k-1)
\zeta_N(2s)R_{f,g}(s),
$$ 
where 
$$
\zeta_N(s) := \prod_{p \text{ prime} \atop p \nmid N} (1 - p^{-s})^{-1}
$$
has meromorphic continuation to $\C$ with possible 
simple poles at $s=0,1$ with
\begin{equation}\label{res}
\text{Res}_{s=1}R_{f,g}(s) = \frac{12. ~  (4\pi)^{k-1}}{(k-1)!} ~<f,g>.
\end{equation}
If $f, g \in S_k(N)$ with $< f,g> =0$, then $R_{f,g}^{*}(s)$ 
is holomorphic everywhere and is of finite order.
When $f=g \in S_k^{\text{new}}(N)$ is a normalized Hecke eigenform, then
\begin{equation}\label{sym}
R_{f,f}(s) := ~\prod_{p |N}(1 + p^{-s})^{-1} ~\frac{\zeta(s)}{\zeta(2s)} 
~L(\text{sym}^2f, s),
\end{equation}
where $L(\text{sym}^2f, s)$ is the symmetric square $L$-function 
attached to $f$ (see \cite{{IK},{KS}} for details).
We now list the lemmas which will be important for our theorem.

\begin{lem}{\rm[Choie-Kohnen (see page 523 of \cite{CK})]}\label{ac-RS}
Let $f \in  S_{k_1}^{\text{ new}}(N_1), ~g \in S_{k_2}^{\text{ new}}(N_2)$
be normalized Hecke eigenforms with normalized eigenvalues $\lambda_f(n)$ and 
$\lambda_g(n)$ respectively. Also let $N$ be a square-free integer 
such that $N_1, N_2 | N$. If $\delta | \frac{N}{N_1}, ~\delta' | \frac{N}{N_2}$
and $T := \text{ gcd}(\delta, \delta')$, then
\begin{eqnarray}\label{action}
R_{f | V_{\delta}, g | V_{\delta'}} (s) 
&=& \left( \frac{\delta\delta'}{T} \right)^{-s -k+1} \\
&& 
. \prod_{ p | \frac{\delta'}{T}} \frac{\lambda_f(p) - \lambda_g(p)p^{-s}}{1 - p^{-2s}}
~~. \prod_{ p | \frac{\delta}{T}} \frac{\lambda_g(p) - \lambda_f(p)p^{-s}}{1 - p^{-2s}}
R_{f, g}(s), \nonumber
\end{eqnarray}
where 
$$
(f | V_{\delta} )(z) := f(\delta z) \phantom{m} \text{for  } z \in \mathcal{H}.
$$
\end{lem}

\begin{lem}{\rm[Choie-Kohnen (see page 524 of \cite{CK})]}\label{sp-basis}
For square-free $N$, the space $S_k(N)$ is the orthogonal
direct sum of the one-dimensional subspaces spanned by
$$
F:= f | \prod_{ p | \frac{N}{d}} ( 1 + \epsilon_{f,p} p^{k/2} V_p),
$$ 
where $f$ runs over all normalized Hecke eigenforms in $S_k^{\text{ new }}(d)$
for all $d | N$ and $\epsilon_{f,p}$ runs over the signs~$\pm 1$.
\end{lem}

\begin{lem}{\rm[Kohnen-Sengupta \cite{KS}]}
For a square-free integer $N$, let $f \in S_k^{\text{ new}}(N)$ 
be a normalized Hecke eigen form. Then for all $t \in \R$ 
and $1 < \delta  < \frac{1}{\log 2}$, we have
\begin{eqnarray}\label{ebd1}
&& 
\phantom{m}
\left|  L( \text{ sym }^2f, ~ \delta + it) \right|
~~\ll~~
\frac{\delta}{(\delta -1)^3}  \nonumber \\
&& 
\phantom{m} \text{and } \phantom{m}
\left|  L( \text{ sym }^2f,  ~1 - \delta + it) \right|
~~\ll~~
\frac{\delta}{(\delta-1)^3}.~
(kN)^{2\delta-1} . ~|1 + it|^{3(\delta - \frac{1}{2})}.
\nonumber
\end{eqnarray}
\end{lem}

As an immediate Corollary, we have
\begin{cor}\label{bd1}
For a square-free integer $N$, let $f \in S_k^{\text{ new}}(N)$ 
be a normalized Hecke eigen form. Then for all $t \in \R$
and $\Delta = \frac{1}{100}$, we have
\begin{eqnarray}\label{ebd1}
&& 
\phantom{m}
\left|  L( \text{ sym }^2f, ~~ 1 + \Delta + it) \right|
~~\ll~~
1 \\
&& 
\phantom{m} \text{and } \phantom{m}
\left|  L( \text{ sym }^2f,  -\Delta + it) \right|
~~\ll~~
(kN)^{1+ 2\Delta} ~|1 + it|^{\frac{3}{2} (1 + 2 \Delta)}.
\nonumber
\end{eqnarray}
\end{cor}

We now state a result of Rademacher which plays
a pivotal role in our work.
\begin{prop}{\rm[Rademacher \cite{HR}]}\label{rad}
Suppose that $g(s)$ is continuous on the closed strip 
$a \le \sigma \le b$ and holomorphic and of finite order
on $a < \sigma < b$. Further, suppose that 
$$
|g(a + it)| \le E|P + a + it|^{\alpha},
\phantom{m}
|g(b + it)| \le F|P + b + it|^{\beta},
$$
where $E, F$ are positive constants and $P, \alpha$ and $\beta$
are real constants that satisfy
$$
P + a > 0, 
\phantom{m}
\alpha \ge \beta.
$$
Then for all $a < \sigma < b$ and $ t \in \R$, we have
$$
|g(\sigma + it)| \le 
(E|P+ \sigma + it |^{\alpha})^{\frac{b - \sigma}{b-a}} 
(F|P+ \sigma + it |^{\beta})^{\frac{\sigma - a}{b-a}}. 
$$
\end{prop}

\begin{prop}\label{sym-equal}
For square-free integer $N$, let
$f \in S_k^{\text{new}}(N)$ be a normalized Hecke eigen form. 
Then for all  $- \Delta < \sigma < 1 + \Delta$ with 
$\Delta = \frac{1}{100}$ and for any $t \in \R$, we have
\begin{eqnarray*}
\left| L\left( \text{sym }^2f, ~~ \sigma + it \right)\right|
&\ll&
(kN)^{(1 + \Delta - \sigma)}. ~(3+ |t|)^{ \frac{3}{2}(1 + \Delta - \sigma)}~.
\end{eqnarray*}
In particular, for all $t \in \R$, we have
\begin{eqnarray}\label{rs-1}
\left| L\left( \text{sym }^2f, ~~ \frac{3}{4} + it \right)\right|
&\ll&
(kN)^{\frac{1}{4} + \Delta} . ~(3+ |t|)^{\frac{3}{8} + \frac{3\Delta}{2}}~.
\end{eqnarray}
\end{prop}

\begin{proof}
Applying Proposition \ref{rad} 
with
\begin{eqnarray*}
&&
a = -\Delta, \phantom{mm} 
b = P = 1+ \Delta, \\
&&
E = E_1 (kN)^{1+ 2\Delta}, \phantom{m}
\alpha = \frac{3}{2}(1+ 2\Delta), 
\phantom{m}
 \beta =0
\end{eqnarray*}
and $F, E_1$ are absolute constants and finally
using Corollary \ref{bd1}, we get our result.
\end{proof}

We can now deduce the following corollary.
\begin{cor}\label{rs-2}
For square-free integer $N$, let
$f \in S_k^{\text{new}}(N)$ be a normalized Hecke eigen form. 
Then for all $t \in \R$ and any $\epsilon > 0$, we have
$$
\left| R_{f,f} \left( \frac{3}{4} + it \right)\right|
~~~~\ll_{\epsilon}~~~~~~
(kN)^{\frac{1}{4} + \Delta} .~e^{(c ~\sqrt{\frac{\log(N+1)}{\log\log(N+2)}})}
. ~(3+ |t|)^{\frac{11}{24} + \frac{3\Delta}{2} + \epsilon},
$$
where $c > 0$ is an absolute constant and $\Delta = \frac{1}{100}$.
\end{cor}
\begin{proof}
Recall that
$$
R_{f,f}(s)  = \prod_{ p \text{ prime } \atop  p|N} (1 + p^{-s})^{-1} 
~.~\frac{\zeta(s)}{\zeta(2s)}.
~L(\text{sym}^2f, s).
$$
To complete the proof, we use Proposition \ref{sym-equal}
in addition to the results that
$$
\left|\zeta (\frac{3}{4} +  it) \right| \ll_{\epsilon} 
|1 + it |^{\frac{1}{12} + \frac{\epsilon}{2}},
\phantom{m} 
\left|\frac{1}{\zeta (\frac{3}{2} + \frac{3}{4} it)} \right| 
\ll_{\epsilon}  |1 + it|^{\frac{\epsilon}{2}}
$$ 
and for $\sigma = 3/4$
$$
\prod_{p \text{ prime} \atop p|N} (1 + p^{-s})^{-1} 
~\ll~
\prod_{p \text{ prime} \atop p|N} (1 + p^{-1/2})^{-1} 
~\ll~
e^{(c ~\sqrt{\frac{\log(N+1)}{\log\log(N+2)}})},
$$
where $c >0$ is an absolute constant. For a
proof of the last inequality, see page 180
of \cite{KS}. 
\end{proof}

In fact, we can prove the following general statement. 
\begin{prop}\label{equal}
For a square-free integer $N$, let
$f \in S_k^{\text{new}}(N)$ be a normalized 
Hecke eigen form. 
Also let  $\epsilon > 0$. Then for $\frac{1}{2}  \le \sigma < 1 + \Delta$ with 
$\Delta = \frac{1}{100}$
and $t\in \R$ with $|t| \gg 1$, we have
\begin{eqnarray*}
\left| R_{f,f} \left( \sigma + it \right)\right|
&\ll_{\epsilon}&
(kN)^{(1 + \Delta - \sigma)} .~~e^{(c  ~\sqrt{\frac{\log(N+1)}{\log\log(N+2)}})}. 
 ~(3+ |t| )^{\frac{3}{2}(1+ \Delta - \sigma)  
 + \text{max }\{\frac{1 - \sigma}{3}, ~0 \} ~+~ \epsilon}
\end{eqnarray*}
where $c  > 0$ is an absolute constant. Further for any 
$t \in \R, \epsilon>0$, we have
\begin{eqnarray*}
\left| R_{f,f} \left( \frac{1}{2} + \frac{\Delta}{2} + it \right)\right|
&\ll_{\epsilon}&
(kN)^{\frac{1 + \Delta}{2}}~ .~~e^{(c  ~\sqrt{\frac{\log(N+1)}{\log\log(N+2)}})}. 
 ~(3+ |t| )^{\frac{3}{4}(1+ \Delta) + \frac{33}{200} ~+~ \epsilon}
\end{eqnarray*}
where $c  > 0$ is an absolute constant. 
\end{prop}

\begin{proof}
Note that for $\frac{1}{2}  \le \sigma < 1 + \Delta$, one has
\begin{equation}\label{e1bd}
\prod_{p \text{ prime} \atop p|N} \left(1 + p^{-s} \right)^{-1} 
~~~\ll~~~
e^{(c ~\sqrt{\frac{\log(N+1)}{\log\log(N+2)}} )}~,
\end{equation}
where $c >0$ is an absolute constant. We now use
the following 
estimates (see pages 145-146 of \cite{GT})
$$
\left|\zeta (\sigma +  it) \right| \ll_{\epsilon} |t|^{\frac{1 - \sigma}{3} 
+ \frac{\epsilon}{2}}
\phantom{m} \text{ and }\phantom{m}
\left|\frac{1}{\zeta (2\sigma + 2it)} \right| 
\ll_{\epsilon} |t|^{\frac{\epsilon}{2}}
$$ 
for $t\in \R$ with $|t| \gg 1$ and $1/2 \le \sigma \le 1$
to complete the proof of the first part. 
For the second part of the Proposition, we proceed
exactly as in Corollary \ref{rs-2} and use the
estimates
$$
\left|\zeta (\frac{1}{2}(1+ \Delta) +  it) \right| 
\ll_{\epsilon} |t|^{\frac{33}{200} 
+ \epsilon}
\phantom{m} \text{ and }\phantom{m}
\left|\frac{1}{\zeta ( 1 + \Delta + 2it)} \right| 
\ll 1
$$ 
for any $t\in \R$.
\end{proof}

Now we would like to estimate the Rankin-Selberg
$L$-function for two distinct newforms $f$ and $g$. 
Here we have the following Proposition. 
\begin{prop}\label{ran}
For square-free integers $N_1, N_2$, let $f \in S_k^{\text{ new}}(N_1)$, 
$g \in S_k^{\text{ new }}(N_2)$ be normalized Hecke eigen forms
with $f \ne g$. Then for any $t \in \R, ~\Delta = \frac{1}{100}$
and $N := \text{lcm} (N_1, N_2)$, we have
\begin{eqnarray}\label{bound}
\left| \zeta_N(2 + 2\Delta + 2it)~.~ R_{f, g} (~ 1+ \Delta + it ~) \right|
&\ll& 1 \nonumber \\
\phantom{m} \text{and } \phantom{m}
\left|\zeta_N(- 2\Delta + 2it)~.~ R_{f, g} ( - \Delta + it) \right|
&\ll &
\tau(N).~ (kN)^{1+ 2\Delta}.~N^{\Delta}
 .~ |1 + it |^{2 + 4\Delta}, \nonumber
\end{eqnarray}
where $\tau(N)$ is the number of divisors of $N$
and for $s \in \C$ with $\Re(s) > 1$, one defines
$$
\zeta_N(s) 
= 
\prod_{ p \text{ prime} \atop p \nmid N} 
( 1 - p^{-s})^{-1} ~.
$$
\end{prop}

\begin{proof}
The first inequality follows as both $\zeta_N(2 + 2\Delta + 2it)$ and
$R_{f, g} (~ 1+ \Delta + it ~)$ are absolutely convergent. Using
the functional equation (see \cite{AO} for details)
$$
\prod_{p \text{ prime} \atop p |D} (1 - 
w_p v_p p^{-s})^{-1} ~R^*_{f,g}(s)
 ~=~
 \prod_{p \text{ prime} \atop p |D} (1-w_pv_p 
 p^{-(1-s)})^{-1} ~R^*_{f,g}(1-s),
$$
we have
\begin{eqnarray*}
\zeta_N(2 - 2s) ~.~ R_{f,g}(1- s)
 &=&
 (2\pi)^{2 - 4s}~. ~N^{2s -1} 
 ~.~\frac{\Gamma(s)}{ \Gamma(1-s)}~. 
 ~\frac{\Gamma(k-1 +s)}{\Gamma(k-s)} \\
 &&
 ~.~  \prod_{p \text{ prime} \atop p |D} \frac{(1-w_pv_p 
 p^{s-1})}{(1 -w_p v_p p^{-s})}
~ .~ \zeta_N(2s)~. ~ R_{f,g}(s) ~, 
\end{eqnarray*}
where $D := \text{ gcd}(N_1, N_2)$ and $w_p, v_p \in \{\pm 1 \}$ are
the eigenvalues of $f$ and $g$ respectively for the corresponding 
Atkin-Lehner involutions.  Now using Stirlings formula (see page 178 
of \cite{KS}), one has
\begin{eqnarray*}
\left| \frac{\Gamma(k + \Delta + it)}{ \Gamma(k -1-\Delta + it)}\right|
&\ll& 
k^{1+ 2\Delta} ~|1+ it|^{1 + 2\Delta}, \\
\text{and} \phantom{mm}
\left|\frac{\Gamma(1 + \Delta + it)}{ \Gamma(-\Delta + it)} \right| 
&\ll&
 | 1+ it|^{1+ 2\Delta} \phantom{mm} \text{ for all } t\in \R.
\end{eqnarray*}
Also for all $t \in \R$, one has
\begin{eqnarray*}
| \prod_{p \text{ prime} \atop p |D} 
(1 -w_p v_p p^{-1 - \Delta + it})^{-1}|
 &\le&
 \prod_{p \text{ prime} \atop p |D} (1 - \frac{1}{p^{1+ \Delta}})^{-1}
~~\ll~  1, \\
 \text{and} \phantom{m}
 | \prod_{p \text{ prime} \atop p |D} (1-w_pv_p p^{\Delta + it})|
&\ll&
 \prod_{p \text{ prime} \atop p |D} (1 + p^{\Delta})
 ~~\ll~~ 
  \prod_{p \text{ prime} \atop p |D} 2p^{\Delta}
  ~=~ \tau(D).~D^{\Delta}~.
\end{eqnarray*}
Hence we have
\begin{eqnarray*}
|\zeta_N(- 2\Delta + 2it)~.~ R_{f, g} ( - \Delta + it) |
&\ll &
\tau(N).~ (kN)^{1+ 2\Delta}. ~N^{\Delta}~
 .~ |1 + it|^{2 + 4\Delta}~.
\end{eqnarray*}
This completes the proof of the Proposition.
\end{proof}

Now using Proposition \ref{rad} and Proposition \ref{ran}, 
we get 
\begin{prop}\label{diff}
For square-free integers $N_1, N_2$, let
$f \in S_k^{\text{new}}(N_1), ~g \in S_k^{\text{new}}(N_2)$
be normalized Hecke eigen forms with $f \ne g$ 
and $N := \text{ lcm }(N_1, N_2)$. Then for any
$t \in \R$ and $- \Delta < \sigma < 1 + \Delta$
with $\Delta = \frac{1}{100}$, we have
$$
\left| \zeta_N(2\sigma + 2it) ~.~ R_{f, g} \left( \sigma + it \right)\right|
~\ll~~~
(kN)^{1+ \Delta - \sigma} .~(\tau(N).~N^{\Delta})^{\frac{1+\Delta 
- \sigma}{1+ 2\Delta}} 
~.~(3 + |t|)^{2(1+\Delta - \sigma)}.
$$
In particular for any $t \in \R, ~\epsilon>0$ and 
$\frac{1}{2}  \le \sigma < 1 + \Delta$, we have
\begin{eqnarray}\label{rq}
\left|  R_{f, g} \left( \sigma + it \right)\right|
&\ll_{\epsilon}&
 (kN)^{1+ \Delta - \sigma} .~(\tau(N).~N^{\Delta})^{\frac{1
 +\Delta - \sigma}{1+ 2\Delta}}~. \\
&&
\phantom{mmmmm}
~~ \log\log(N+2)
~.~(3 + |t|)^{2(1+\Delta - \sigma) + \epsilon}. \nonumber
\end{eqnarray}
\end{prop}

\begin{proof}
 From the theory of newforms, we know that
$<f, g> = 0$. In this case, we know that $\zeta_N(2s)R_{f,g}(s)$ is entire 
and  is of finite order (see \cite{{WL},{AO}} for details). Set
\begin{eqnarray*}
&&
a = -\Delta~, 
\phantom{mm}
b = P = 1+ \Delta, \\
&&
E =  E_1(kN)^{1+ 2\Delta} .~\tau(N).~N^{\Delta}, \phantom{m}
\alpha =  2 + 4\Delta, \phantom{m}
\beta = 0
\end{eqnarray*}
and $F, E_1$ are absolute constants.  We now use Proposition \ref{ran} 
and apply Proposition \ref{rad} in the strip $a \le \sigma \le b$
to conclude the first part of our result. 
Further note that for $\frac{1}{2} \le \sigma < 1+ \Delta$,
one has
$$
| \zeta_N(2\sigma + 2it)^{-1} |
~~\ll~~
\frac{ \prod_{ p \text{ prime } \atop p|N} (1 + p^{-1}) }
{|\zeta(2\sigma + 2it) |}
~~\ll_{\epsilon}~~
\log\log(N+2)~.~| 1+ it |^{\epsilon}~.
$$
This completes the proof of \eqref{rq}.
\end{proof}

As an immediate Corollary, we have
\begin{cor}\label{rs1}
For square-free integers $N_1, N_2$, let
$f \in S_k^{\text{new}}(N_1), ~g \in S_k^{\text{new}}(N_2)$
be normalized Hecke eigen forms with $f \ne g$ 
and $N := \text{ lcm }(N_1, N_2)$. Then for any
$t \in \R$ and $\Delta = \frac{1}{100}$, we have 
\begin{eqnarray*}
\left|  R_{f, g} \left( \frac{1}{2} + \frac{\Delta}{2} + it \right)\right|
&\ll_{\epsilon}&
 (kN)^{\frac{1}{2}(1+ \Delta)}.~(\tau(N).~N^{\Delta}
 )^{\frac{(1+ \Delta)}{2(1+ 2\Delta)}}~.
~\log\log(N+2)~.~~(3 + |t| )^{ 1 + \Delta + \epsilon} \\
\left|  R_{f, g} \left( \frac{3}{4} + it \right)\right|
&\ll_{\epsilon}&
 (kN)^{\frac{1}{4} + \Delta}.~(\tau(N).~N^{\Delta})^{\frac{(1
 + 4\Delta)}{4(1+ 2\Delta)}}~.
~\log\log(N+2)~.~~(3 + |t| )^{\frac{1}{2} + 2\Delta + \epsilon}. \nonumber
\end{eqnarray*}
\end{cor}

Combining Proposition \ref{diff} and Proposition \ref{equal}, 
we now have the following statement.
\begin{prop}\label{any}
For a square-free integer $N$ and integers $N_1, N_2 | N$, let
$f \in S_k^{\text{new}}(N_1), ~g \in S_k^{\text{new}}(N_2)$
be normalized Hecke eigen forms, not necessarily distinct. 
Also let  $\epsilon > 0$. Then 
for $\frac{1}{2}  \le \sigma  < 1 + \Delta$ 
with $\Delta = \frac{1}{100}$ and $t \in \R$ with $|t| \gg 1$, 
we have
\begin{eqnarray*}
\left|  
R_{f, g} \left( \sigma + it \right) 
\right|
&\ll_{\epsilon} &
(kN)^{1+ \Delta - \sigma}.~(\tau(N).~N^{\Delta}
)^{\frac{1+ \Delta - \sigma}{1+ 2\Delta}}. 
 ~e^{(c_1 ~\sqrt{\frac{\log(N+1)}{\log\log(N+2)}} )}~.~
(3 + |t| )^{ 2(1+ \Delta - \sigma) ~+~ \epsilon},
\end{eqnarray*}
where $c_1 > 0$ is an absolute constant.
\end{prop}

Combining Proposition \ref{equal}, Corollary \ref{rs-2} and 
Corollary \ref{rs1}, we now get
\begin{prop}\label{half}
For a square-free integer $N$ and integers $N_1, N_2 | N$, let
$f \in S_k^{\text{new}}(N_1), ~g \in S_k^{\text{new}}(N_2)$
be normalized Hecke eigen forms, not necessarily distinct. 
Then for any $\epsilon > 0, ~ t \in \R, ~\Delta = \frac{1}{100}$, 
we have
\begin{eqnarray*}
\left| R_{f,g} \left( \frac{1}{2} + \frac{\Delta}{2} + it \right)\right|
&\ll_{\epsilon}&
 (kN)^{\frac{1}{2}(1+ \Delta)}.~(\tau(N).~N^{\Delta}
 )^{\frac{1+ \Delta}{2(1+ 2\Delta)}}~
. ~e^{(c_1 ~\sqrt{\frac{\log(N+1)}{\log\log(N+2)}} )}
.~~(3 + |t| )^{ 1 + \Delta + \epsilon} \\
\text{and} \phantom{mm}
\left| R_{f,g} \left( \frac{3}{4}  + it \right)\right|
&\ll_{\epsilon}&
 (kN)^{\frac{1}{4} + \Delta}.~(\tau(N).~N^{\Delta}
 )^{\frac{1+ 4 \Delta}{4(1+ 2\Delta)}}~
. ~e^{(c_1 ~\sqrt{\frac{\log(N+1)}{\log\log(N+2)}} )}
.~~(3 + |t| )^{ \frac{1}{2} + 2 \Delta + \epsilon}, \nonumber
\end{eqnarray*}
where $c_1 > 0$ is an absolute constant.
\end{prop}

Further, we have the following Proposition.
\begin{prop}\label{rs-basis}
For square-free integer $N$ and $d, d' | N$, let
$f \in S_k^{\text{new}}(d), ~g \in S_k^{\text{new}}(d')$
be normalized Hecke eigen forms.  
For $\epsilon_{f, p}, \epsilon_{g,p} \in \{ \pm 1\}$, define
$$
F:= f ~|~ \prod_{ p | \frac{N}{d}} ( 1 + \epsilon_{f,p} p^{k/2} V_p)
\phantom{m}\text{and}\phantom{m}
G:= g ~|~ \prod_{ p | \frac{N}{d'}} ( 1 + \epsilon_{g,p} p^{k/2} V_p)
$$
as in Lemma \ref{sp-basis}. Also let  $\epsilon > 0$. 
Then for $\frac{1}{2}  \le \sigma < 1 + \Delta$ with
$\Delta = \frac{1}{100}$ and $t \in \R$ with 
$|t| \gg 1$, we have
\begin{eqnarray*}
\left|  
R_{F, G} \left(\sigma + it \right) 
\right|
~~\ll_{\epsilon}~~
(kN)^{1 + \Delta - \sigma} 
~.~N^{\frac{\Delta(1 + \Delta - \sigma)}{1+ 2\Delta} + \frac{1}{2}} 
~.~e^{(c_2 ~\frac{\log(N+1)}{\log\log(N+2)})}
~. ~ (3 + |t| )^{ 2(1+ \Delta -\sigma) +  \epsilon},
\end{eqnarray*}
where $c_2 > 0$ is an absolute constant. Further,
we have for all $t \in \R$
\begin{eqnarray*}
\left|  
R_{F, G} \left(\frac{1}{2} (1 + \Delta) + it \right) 
\right|
&\ll_{\epsilon}&
(kN)^{\frac{1}{2}(1 + \Delta)} 
~.~N^{\frac{\Delta(1 + \Delta)}{2(1+ 2\Delta)} + \frac{1}{2}}
~.~e^{(c_2 ~\frac{\log(N+1)}{\log\log(N+2)})}
~. ~ (3 + |t| )^{1+ \Delta  +  \epsilon}, \\
\text{and}\phantom{m}
\left|  
R_{F, G} \left(\frac{3}{4} + it \right) 
\right|
&\ll_{\epsilon}&
(kN)^{\frac{1}{4} + \Delta} 
~.~N^{\frac{\Delta(1 + 4\Delta)}{4(1+ 2\Delta)} + \frac{1}{2}}
~.~e^{(c_2 ~\frac{\log(N+1)}{\log\log(N+2)})}
~. ~ (3 + |t| )^{\frac{1}{2} + 2\Delta  +  \epsilon},
\end{eqnarray*}
where the constant $c_2>0$ is absolute.
\end{prop}

\begin{proof}
Note that 
$$
R_{F, G} \left(\sigma + it \right)  
=
\sum_{\delta |\frac{N}{d}, ~ \delta' | \frac{N}{d'}}
 \epsilon_{f, \delta} \epsilon_{g, \delta'} ~(\delta \delta')^{k/2}
 ~R_{f|V_{\delta}, g|V_{\delta'}}(\sigma + it),
$$
where
$$
\epsilon_{f, \delta} = \prod_{p | \delta} \epsilon_{f, p}
\phantom{mm}\text{and}\phantom{mm}
\epsilon_{g, \delta'} = \prod_{p | \delta'} \epsilon_{g, p}.
$$
We would now like to find bounds for 
$R_{f|V_{\delta}, g|V_{\delta'}}(\sigma + it)$
when $1/2 \le \sigma < 1 + \Delta$ 
and any $t \in \R$ with $|t| \gg 1$.
For $T := \text{ gcd }(\delta, \delta')$, we have
\begin{eqnarray*}
\prod_{p \text{ prime } \atop p | \frac{\delta'}{T}} |\lambda_f (p) - 
\lambda_g (p) p^{-\sigma -it}|
&\ll&
 4^{\nu( \frac{\delta'}{T})}  \\
\phantom{m} \text{and}\phantom{mmm}
 \prod_{p \text{ prime} \atop p | \frac{\delta}{T}} |\lambda_g (p) - 
\lambda_f (p) p^{-\sigma -it}|
&\ll&
 4^{\nu( \frac{\delta}{T})}~,
\end{eqnarray*}
where $\nu(N)$ is the number of distinct prime 
factors of $N$.  Also
\begin{eqnarray*}
\left| \prod_{p \text{ prime} \atop p | N} (1 - p^{-2\sigma -2it})^{-1} \right|
~\le~ 
\prod_{p \text{ prime} \atop p | N} (1 - p^{-2\sigma})^{-1} 
&\le&
\prod_{p \text{ prime} \atop p | N} (1 - p^{-1})^{-1} \\
&\ll& 
 \prod_{p \text{ prime} \atop p | N} (1 + p^{-1})
~~\ll~~
\log\log(N+2).
\end{eqnarray*}
Hence by applying Lemma \ref{ac-RS}, we get
$$
\left| R_{f|V_{\delta}, g|V_{\delta'}}(\sigma + it) \right|
~~\ll~~
\left(\frac{\delta\delta'}{T}\right)^{-k + \frac{1}{2}}~
.~ 4^{\nu(N)}~.~\log\log(N+2)~.~ |R_{f,g} (\sigma + it)| 
$$
for all $\frac{1}{2} \le \sigma < 1 + \Delta$ 
with $\Delta = \frac{1}{100}$ 
and any $t \in \R$ with $|t| \gg 1$.
Thus for $\sigma$ and $t$ in the above mentioned range
and for any $\epsilon >0$, we get
\begin{eqnarray*}
\left|  
R_{F, G} \left(\sigma + it \right) 
\right|
&\ll_{\epsilon} &
4^{\nu(N)}~.~(kN)^{1 + \Delta - \sigma} 
~.~(\tau(N).~N^{\Delta})^{\frac{1 + \Delta - \sigma}{1+ 2\Delta}}
~.~\log\log(N+2)~.~
~~e^{(c_1 ~\sqrt{\frac{\log(N+1)}{\log\log(N+2)})}}~. \\
&&
\phantom{mmm}
( \sum_{ \delta | \frac{N}{d}, \delta' | \frac{N}{d'}} 
\left(\frac{\delta\delta'}{T^2}\right)^{-\frac{k}{2} 
+  \frac{1}{2}} \sqrt{T} ~)~~.~~
~ (3 + |t| )^{ 2 (1+ \Delta -\sigma) +  \epsilon} \\
&\ll_{\epsilon} &
4^{\nu(N)}~.~(kN)^{1 + \Delta - \sigma} 
~.~(\tau(N).~N^{\Delta})^{\frac{1 + \Delta - \sigma}{1+ 2\Delta}}
~.~\log\log(N+2)~. ~
~e^{(c_1 ~\sqrt{\frac{\log(N+1)}{\log\log(N+2)})}}~.\\
&&
\phantom{mmm}
 ~( \sum_{ \delta | \frac{N}{d}, \delta' | \frac{N}{d'}} 
\sqrt{T} ~)~. ~~
~ (3 + |t| )^{ 2(1+ \Delta -\sigma) +  \epsilon},
\end{eqnarray*}
by noting that 
$$
\left(\frac{\delta\delta'}{T^2}\right)^{-\frac{k}{2} +  \frac{1}{2}} 
~\le~ 1.
$$
Finally using the inequality (see page 533 of \cite{CK})
$$
(\sum_{ \delta | \frac{N}{d}, \delta' | \frac{N}{d'}} 
\sqrt{\text{ gcd }(\delta, \delta')})
~~\le~~ 
\sqrt{N}.~ \tau^2(N),
$$
where $\tau(N)$ is the number of divisors
of $N$ and 
\begin{equation*}
4^{\nu(N)}.~\log\log(N+2).
~ e^{(c_1 \sqrt{\frac{\log(N+1)}{\log\log(N+2)}} )}.
~\tau(N)^{\frac{1 + \Delta - \sigma}{1+ 2\Delta}}
.~(\sum_{ \delta | \frac{N}{d}, 
\delta' | \frac{N}{d'}} \sqrt{\text{gcd} (\delta, \delta')} )
~~~\ll~~
 \sqrt{N}.~ e^{(c_2 \frac{\log(N+1)}{\log\log(N+2)})}
\end{equation*}
where $c_2>0$ is an absolute constant, we get
for any $\epsilon>0, ~1/2  \le \sigma < 1 + \Delta$ 
and $t\in \R$ with $|t| \gg 1$ that
$$
\left|  
R_{F, G} \left(\sigma + it \right) 
\right|
~~\ll_{\epsilon}~~
(kN)^{1 + \Delta - \sigma} 
~.~N^{\frac{\Delta(1 + \Delta - \sigma)}{1+ 2\Delta}}.~\sqrt{N} 
~.~e^{(c_2 ~\frac{\log(N+1)}{\log\log(N+2)})}
~. ~ (3 + |t| )^{ 2(1+ \Delta -\sigma) +  \epsilon}.
$$
This completes the proof of the first part of the theorem. 
Proceeding as in the first part and applying Proposition
\ref{half}, we get for any $t\in \R$ and $\epsilon>0$
that
\begin{eqnarray*}
\left|  
R_{F, G} \left(\frac{1}{2} (1 + \Delta) + it \right) 
\right|
&\ll_{\epsilon}&
(kN)^{\frac{1}{2}(1 + \Delta)} 
~.~N^{\frac{\Delta(1 + \Delta)}{2(1+ 2\Delta)}}
.~\sqrt{N} ~.~e^{(c_2 ~\frac{\log(N+1)}{\log\log(N+2)})}
~. ~ (3 + |t| )^{1+ \Delta  +  \epsilon}, \\
\text{and}\phantom{m}
\left|  
R_{F, G} \left(\frac{3}{4} + it \right) 
\right|
&\ll_{\epsilon}&
(kN)^{\frac{1}{4} + \Delta} 
~.~N^{\frac{\Delta(1 + 4\Delta )}{4(1+ 2\Delta)}}.~\sqrt{N} 
~.~e^{(c_2 ~\frac{\log(N+1)}{\log\log(N+2)})}
~. ~ (3 + |t| )^{\frac{1}{2} + 2\Delta  +  \epsilon},
\end{eqnarray*}
where the constant $c_2 >0$ is absolute.
This completes the proof of the Proposition.
\end{proof}

We also need the following lemma of Choie and Kohnen.

\begin{lem}{\rm[Choie-Kohnen (see page 534 of \cite{CK})]}\label{ld}
Let $N$ be square-free and $d|N$. Also let $f \in~S_k^{\text{ new}}(d)$
be a normalized Hecke eigen form and $F$ be as in Lemma \ref{sp-basis}.
Then
$$
1 \ll \frac{(4 \pi)^{k-1}}{(k-1)!}~. \log(kN) ~. \prod_{ p|N}(1 + \frac{1}{p})
~.e^{( \tilde{c} \sqrt{\frac{\log(N+1)}{\log\log(N+2)}})}~. <F, F>,
$$
where $\tilde{c}  > 0$ is an absolute constant.
\end{lem}

Finally, we have the following Propositions which
are important to complete the proof of Theorem
\ref{thmthree}.

\begin{prop}\label{firstle}
For a square-free $N$, let $f$ be a non-zero cusp form of weight 
$k$ for $\Gamma_0(N)$ with real Fourier coefficients. Let
$$
D(k, N) := \frac{2 \pi^2 (4\pi)^{k-1}}{(k-1)!} 
\prod_{ p| N \atop p \text{ prime }} ( 1 + 1/p)
$$
and
\begin{equation}\label{norm}
\tilde{f} := \frac{f}{ \sqrt{D(k,N)} || f ||},
\end{equation}
where $|| . ||$ is the Petersson norm.
If for $n \in \N$, $\lambda_{\tilde{f}}(n)$ are normalized Fourier
coefficients of $\tilde{f}$, then for any 
$0< \epsilon < \frac{12}{25}$ and $\Delta = \frac{1}{100}$, we have
\begin{eqnarray}\label{eq-r2}
&&
\sum_{n \leq x} \lambda_{\tilde{f}}^2(n) \log\left(\frac{x}{n}\right) \\
&=&
\frac{6}{\pi^2} \prod_{ p \text{ prime} \atop p |N} (1 + \frac{1}{p})^{-1} ~x 
~+~  O_{\epsilon} \left( (kN)^{\frac{5}{4} + \Delta}
~.~N^{\frac{\Delta(1 + 4\Delta )}{4(1+ 2\Delta)}}.~\sqrt{N} 
~.~e^{(c_3 ~\frac{\log(N+1)}{\log\log(N+2)})}
~. ~\log k~.~x^{3/4} \right), \nonumber
\end{eqnarray}
where $c_3 > 0$ is an absolute constant. 
\end{prop}

\begin{proof}
Note that if $a_{\tilde{f}}(n)$ are the Fourier coefficients
of $\tilde{f}$, then $\lambda_{\tilde{f}}(n) :=  
a_{\tilde{f}}(n)/n^{\frac{k-1}{2}}$.
We use Perron's formula to write
$$
\sum_{n \le x} \lambda_{\tilde{f}}^2(n) \log\left(\frac{x}{n}\right)
~=~
\frac{1}{2 \pi i} \int_{1 + \frac{\Delta}{2}
 - i\infty}^{1+ \frac{\Delta}{2} + i \infty} 
R_{\tilde{f}, \tilde{f}}(s) ~\frac{x^s}{s^2} ~ds,
$$
where $R_{\tilde{f}, \tilde{f}}(s)$ is the Rankin-Selberg $L$-function
attached to $\tilde{f}$. Then shifting the line of integration and 
using equations \eqref{res} and \eqref{norm}, we get
$$
\sum_{n \leq x} \lambda_{\tilde{f}}^2(n) \log\left(\frac{x}{n}\right)
~=~
\frac{6}{\pi^2} \prod_{ p |N \atop \text{ prime }} (1 + 1/p)^{-1} ~x
~+~ 
\frac{1}{2\pi i} \int_{\frac{3}{4}  
- i\infty}^{\frac{3}{4} + i \infty}
R_{\tilde{f}, \tilde{f}}(s) ~\frac{x^s}{s^2} ~ds.
$$
We write
$$
\tilde{f} := \sum_{\tau=1}^{d_{k,N}} \alpha_{\tau} F_{\tau},
$$
where $d_{k,N} := \text{ dim }S_k(N)$ and
$\{ F_{\tau} \}_{1\le \tau \le d_{k, N}}$ is the special
orthogonal basis of $S_k(N)$ (in some fixed order) 
mentioned in Lemma \ref{sp-basis}. Set
$$
A_{\tilde{f}} :=  \sum_{\tau =1}^{d_{k,N}} |\alpha_{\tau}|.
$$
Since 
$$
R_{\tilde{f}, \tilde{f}}(s) 
~=~ 
\sum_{\tau , \tau'} \alpha_{\tau} .\overline{\alpha}_{\tau'} 
~R_{F_{\tau}, F_{\tau'}}(s),
$$
using Proposition \ref{rs-basis}, we get
\begin{eqnarray*}
\left|  
R_{F, G} \left(\frac{3}{4} + it \right) 
\right|
&\ll_{\epsilon}&
(kN)^{\frac{1}{4} + \Delta}
~.~N^{\frac{\Delta(1 + 4\Delta )}{4(1+ 2\Delta)}}
.~\sqrt{N}
~.~e^{(c_2 ~\frac{\log(N+1)}{\log\log(N+2)})}
~. ~ (3 + |t| )^{\frac{1}{2} + 2\Delta  +  \epsilon}
\end{eqnarray*}
for all $t \in \R$.  We know by 
Chebyshef's inequality that
$$
A_{\tilde{f}}^2 ~~\le~~~ d_{k, N} .~
\sum_{ \tau =1}^{d_{k,N}} |\alpha_{\tau}|^2
$$
and hence for any $t\in \R$
\begin{eqnarray*}
\left| R_{\tilde{f}, \tilde{f}}( \frac{3}{4} + it ) \right|
&\ll_{\epsilon}&
d_{k, N} ~.~
(kN)^{\frac{1}{4} + \Delta} 
~.~N^{\frac{\Delta(1 + 4\Delta )}{4(1+ 2\Delta)}}.~\sqrt{N}
~.~e^{(c_2 ~\frac{\log(N+1)}{\log\log(N+2)})}
~. ~ (3 + |t| )^{\frac{1}{2} + 2\Delta  +  \epsilon}
~. ~\sum_{ \tau =1}^{d_{k,N}} |\alpha_{\tau}|^2.
\end{eqnarray*}
Note that
$$
< \tilde{f}, ~\tilde{f} > 
~~~~~~=~~~~ 
\sum_{\tau = 1}^{d_{k,N}} |\alpha_{\tau}|^2 
< F_{\tau}, ~ F_{\tau}>.
$$
since the basis $\{ F_{\tau} \}_{1 \le \tau \le d_{k, N}}$ 
is orthogonal.
Using this along with equation \eqref{norm} and
Lemma \ref{ld}, we can write for any $t\in \R$
\begin{eqnarray}\label{dar1}
\left| R_{\tilde{f}, \tilde{f}}( \frac{3}{4} + it ) \right|
&\ll_{\epsilon} &
d_{k, N} ~.~
(kN)^{\frac{1}{4} + \Delta}
~.~N^{\frac{\Delta(1 + 4\Delta )}{4(1+ 2\Delta)}}~.~\sqrt{N}  \\
&&
\phantom{mm}
~.~e^{(\tilde{c}_2 ~\frac{\log(N+1)}{\log\log(N+2)})}
~. ~ (3 + |t| )^{\frac{1}{2} + 2\Delta  +  \epsilon}
~.~\frac{(4 \pi)^{k-1}}{(k-1)!}  \nonumber\\
&&
\phantom{mmm}
.~ \log ~k.~ \prod_{ p \text{ prime} \atop p|N}(1 + \frac{1}{p})  
~.~ \sum_{ \tau =1}^{d_{k,N}} |\alpha_{\tau}|^2 
< F_{\tau}, ~ F_{\tau}>.  \nonumber\\
&\ll_{\epsilon}&
d_{k, N}  ~~.~~~
(kN)^{\frac{1}{4} + \Delta}
~.~N^{\frac{\Delta(1 + 4\Delta )}{4(1+ 2\Delta)}}.~\sqrt{N}
~.~e^{(\tilde{c}_2 ~\frac{\log(N+1)}{\log\log(N+2)})}
~. ~\log k~. \nonumber  \\
 && 
~~ (3 + |t| )^{\frac{1}{2} + 2 \Delta + \epsilon}~
.~\frac{(4 \pi)^{k-1}}{(k-1)!}.~ \prod_{p \text { prime} 
\atop  p|N}(1 + \frac{1}{p}).~
 < \tilde{f}, ~ \tilde{f}>. \nonumber \\
 &\ll_{\epsilon}&
 d_{k, N}  ~~.~~~
(kN)^{\frac{1}{4} + \Delta}
~.~N^{\frac{\Delta(1 + 4\Delta )}{4(1+ 2\Delta)}}
.~\sqrt{N} \nonumber \\
&&
\phantom{mmm}
~.~e^{(\tilde{c}_2 ~\frac{\log(N+1)}{\log\log(N+2)})}
~. ~\log k~.
~~ (3 + |t| )^{\frac{1}{2} + 2 \Delta + \epsilon}~, \nonumber
\end{eqnarray}
where the constants $\tilde{c}_2> 0$ is absolute. 
Since
\begin{eqnarray}\label{dar2}
d_{k,N} ~\ll~  k.~N.~ \log\log(N+2)
\end{eqnarray}
and choosing $\epsilon$ with $ 0 <  \epsilon < \frac{12}{25}$, we now have 
\begin{eqnarray*}
&&
\sum_{n \leq x} \lambda_{\tilde{f}}^2(n) \log\left(\frac{x}{n}\right) \\
&=&
\frac{6}{\pi^2} \prod_{ p \text{ prime} \atop p |N} (1 + \frac{1}{p})^{-1} ~x
~+~  O_{\epsilon} \left( (kN)^{\frac{5}{4} + \Delta}
~.~N^{\frac{\Delta(1 + 4\Delta )}{4(1+ 2\Delta)}}
.~\sqrt{N}
~.~e^{(c_3 ~\frac{\log(N+1)}{\log\log(N+2)})}
~. ~\log k~.~x^{3/4} \right),
\end{eqnarray*}
where $c_3 >0$ is an absolute constant.
This completes the proof of the Proposition.
\end{proof}

\begin{prop}\label{secle}
For a square-free $N$, let $f$ be a non-zero cusp form of weight 
$k$ for $\Gamma_0(N)$ with real Fourier coefficients. Also let
$\tilde{f}$ be as in Proposition \ref{firstle} and 
$\lambda_{\tilde{f}}(n)$ be the normalized $n$-th Fourier
coefficient of $\tilde{f}$. We have
\begin{eqnarray}\label{eq-r3}
&&
\sum_{n \leq x} \lambda_{\tilde{f}}^2(n) \\
&=& 
\frac{6}{\pi^2} \prod_{ p |N \atop \text{ prime }} (1 + \frac{1}{p})^{-1} ~x
~+~  
O_{\epsilon} \left( (kN)^{\frac{3}{2} + \frac{\Delta}{2}} 
~.~N^{\frac{\Delta(1 + \Delta)}{2(1+ 2\Delta)}}
~.~\sqrt{N}~.~ \log k
~.~e^{(c_4 ~\frac{\log(N+1)}{\log\log(N+2)})}~.
~~x^{\frac{3}{4} (1 + \Delta)+ \epsilon}
\right), \nonumber
\end{eqnarray}
where $0< \epsilon < 97/400$, $\Delta = \frac{1}{100}$ 
and $c_4>0$ is an absolute constant. 
\end{prop}

\begin{proof}
If $a_{\tilde{f}}(n)$ are the Fourier coefficients
of $\tilde{f}$, then $\lambda_{\tilde{f}}(n) :=  
a_{\tilde{f}}(n)/n^{\frac{k-1}{2}}$.
As in Proposition \ref{firstle}, we write
$$
\tilde{f} := \sum_{\tau=1}^{d_{k,N}} \alpha_{\tau} F_{\tau},
$$
where $d_{k,N} := \text{ dim }S_k(N)$ and
$\{ F_{\tau} \}_{1\le \tau \le d_{k, N}}$ is the special
orthogonal basis of $S_k(N)$ (in some fixed order) 
mentioned in Lemma \ref{sp-basis} and set
$$
A_{\tilde{f}} :=  \sum_{\tau =1}^{d_{k,N}} |\alpha_{\tau}|.
$$
If $F_{\tau} = f | \prod_{ p | \frac{N}{d}} (1 + \epsilon_{f,p} p^{k/2} V_p)$,
where $d | N$ and $\epsilon_{f, \delta} = 
\prod_{ p| \delta} \epsilon_{f,p}$, then
$$
 | \lambda_{F_{\tau}}(n)| = | \sum_{\delta | \frac{N}{d}} 
 \epsilon_{f, \delta}~ \delta^{1/2} 
~ \lambda_f( \frac{n}{\delta}) |
 ~\ll~
( \tau(N).~\sqrt{N})~.~\tau(n),
$$
where $\tau(N)$ is the number of divisors of $N$ 
and $\lambda_f( n/ \delta) =0$ if $\delta \nmid n$. Hence 
\begin{eqnarray*}
|\lambda_{\tilde{f}}(n)|^2 
~~=~~
|\sum_{\tau=1}^{d_{k,N}} \alpha_{\tau} \lambda_{F_{\tau}}(n)|^2 
&\ll&
\sum_{1 \le \tau, \tau'  \le d_{k,N}} |\alpha_{\tau}\alpha_{\tau'}| .~~~
 | \lambda_{F_{\tau}}(n)\lambda_{F_{\tau'}}(n)| \\
&\ll&
\tau^2(N).~ N.~A_f^2.~\tau^2(n) \\
&\ll&
k.~ N^2.~\log k.~ 
e^{(c' \frac{\log(N+1)}{\log\log(N+2)})}.
~n^{\epsilon_1} 
\end{eqnarray*}
for any $\epsilon_1 > 0$ and absolute constant $c' >0$. 
Now by Perron's formula (see page 67 of \cite{RM}) and for
$x \not\in \N$, we have
$$
\sum_{n \le x} \lambda_{\tilde{f}}^2(n) 
~=~
\frac{1}{2 \pi i} \int_{1 + \Delta/2
- iT}^{1 + \Delta/2 + i T} 
R_{\tilde{f}, \tilde{f}}(s) ~\frac{x^s}{s} ~ds
\phantom{m}
+~~
O\left(k~.~ N^2~.~\log k~.~ 
e^{c' \frac{\log(N+1)}{\log\log(N+2)}}
~.~\frac{x^{1 + \Delta/2 + \epsilon}}{T} \right),
$$
where $R_{\tilde{f}, \tilde{f}}(s)$ is the Rankin-Selberg $L$-function
attached to $\tilde{f}$, $\Delta = \frac{1}{100}$ and $1 \le T \le x$
to be chosen later. Then shifting the line of integration and 
using equations \eqref{res} and \eqref{norm}, we get
\begin{eqnarray*}
\frac{1}{2 \pi i} \int_{1 + \frac{\Delta}{2} - iT}^{1 + \frac{\Delta}{2} + i T} 
R_{\tilde{f}, \tilde{f}}(s) ~\frac{x^s}{s} ~ds
&=&
\frac{6}{\pi^2} \prod_{ p |N \atop \text{ prime }} (1 + 1/p)^{-1} ~x
~+~  \frac{1}{2\pi i} 
 \int_{ 1 +  \frac{\Delta}{2}  - iT}^{\frac{1}{2} + \frac{\Delta}{2} - iT} 
R_{\tilde{f}, \tilde{f}}(s) ~\frac{x^s}{s} ~ds \\
&& 
+~  \frac{1}{2\pi i} \int_{\frac{1}{2} +  
\frac{\Delta}{2} - iT}^{\frac{1}{2} +  \frac{\Delta}{2} + iT}
R_{\tilde{f}, \tilde{f}}(s) ~\frac{x^s}{s} ~ds
~+~  \frac{1}{2\pi i} \int_{\frac{1}{2} +  \frac{\Delta}{2} + iT}^{1
 + \frac{\Delta}{2} + iT} 
R_{\tilde{f}, \tilde{f}}(s) ~\frac{x^s}{s} ~ds\\
&=&
\frac{6}{\pi^2} \prod_{ p |N \atop \text{ prime }} (1 + 1/p)^{-1} ~x
~+~  J_1 ~+~ J_2 ~+~ J_3, \phantom{m}\text{say}.
\end{eqnarray*}
Proceeding as in the derivation of equation \eqref{dar1}, we have
for any $t\in \R, ~\epsilon>0$ that
\begin{eqnarray*}
\left| R_{\tilde{f}, \tilde{f}}( \frac{1}{2} + \frac{\Delta}{2} + it ) \right|
&\ll_{\epsilon}&
(kN)^{\frac{3}{2} + \frac{\Delta}{2}} 
~.~N^{\frac{\Delta(1 + \Delta)}{2(1+ 2\Delta)}}
~.~\sqrt{N}
~.~ \log k
~.~e^{(c'' ~\frac{\log(N+1)}{\log\log(N+2)})}
~. ~ (3 + |t| )^{1+ \Delta  +  \epsilon}, 
\end{eqnarray*}
where the constants $c''> 0$ is absolute.
Hence we have
\begin{eqnarray}\label{dui}
|J_2| 
&\ll_{\epsilon}&
(kN)^{\frac{3}{2} + \frac{\Delta}{2}} 
~.~N^{\frac{\Delta(1 + \Delta)}{2(1+ 2\Delta)}}
~.~\sqrt{N}
~.~ \log k
~.~e^{(c'_1 ~\frac{\log(N+1)}{\log\log(N+2)})}~. 
~x^{\frac{1}{2} + \frac{\Delta}{2}}.
~~T^{1+ \Delta + \epsilon}, 
\end{eqnarray}
where $c'_1> 0$ is an absolute constant.
Again we note that for any $1/2  \le \sigma \le 1 + \Delta / 2$ 
and $t \in \R$ with $|t|  \gg 1$, one has
\begin{eqnarray*}
\left| R_{\tilde{f}, \tilde{f}}( \sigma + it ) \right|
&\ll_{\epsilon}&
(kN)^{2 + \Delta - \sigma} 
~.~N^{\frac{\Delta(1 + \Delta - \sigma)}{1+ 2\Delta}}~.~\sqrt{N}
~.~\log k
~.~e^{(c'_2 ~\frac{\log(N+1)}{\log\log(N+2)})}
~. ~ (3 + |t| )^{ 2(1+ \Delta - \sigma) + \epsilon},
\end{eqnarray*}
where the constant $c'_2 > 0$ is absolute. Hence
\begin{eqnarray*}
|J_1 ~+~ J_3 | 
&\ll&
(kN)^{\frac{3}{2} + \frac{\Delta}{2}} 
~.~N^{\frac{\Delta(1 + \Delta)}{2(1+ 2\Delta)}}
~.~\sqrt{N}
~.~ \log k
~.~e^{(c'_2 ~\frac{\log(N+1)}{\log\log(N+2)})}~. \\
&&
\phantom{mmm}
 \frac{1}{T}
 \left\{ \text{max}_{ \frac{1}{2} + \frac{\Delta}{2} \le \sigma \le 1 + \frac{\Delta}{2}}
\left( \frac{x}{T^2}  \right)^{\sigma} 
T^{2 (1+ \Delta) + \epsilon} \right\},
\end{eqnarray*}
where the constant in $\ll$  and the constants $c>0$
are absolute. Now by choosing $T = x^{1/4}$,
we get
\begin{eqnarray*}
&&
\sum_{n \leq x} \lambda_{\tilde{f}}^2(n) \\
&=&
\frac{6}{\pi^2} \prod_{ p |N \atop \text{ prime }} (1 + 1/p)^{-1} ~x 
~+~~
O_{\epsilon} \left( (kN)^{\frac{3}{2} + \frac{\Delta}{2}} 
~.~N^{\frac{\Delta(1 + \Delta)}{2(1+ 2\Delta)}}
~.~\sqrt{N}
~.~ \log k
~.~e^{(c_4 ~\frac{\log(N+1)}{\log\log(N+2)})}~.
~~x^{\frac{3}{4}( 1 + \Delta) + \epsilon}
\right),
\end{eqnarray*}
where $0< \epsilon < 97/400$ and $c_4>0$ is an absolute constant. 
This completes the proof. 
\end{proof}

\begin{prop}\label{thle}
For a square-free $N$, let $f$ be a non-zero cusp form of weight 
$k$ for $\Gamma_0(N)$ with real Fourier coefficients. Also let
$\tilde{f}$ be as in Proposition \ref{firstle} and 
$\lambda_{\tilde{f}}(n)$ be the normalized $n$-th Fourier
coefficient of $\tilde{f}$. We then have
\begin{eqnarray}\label{hecke-1}
\sum_{n \leq x} \lambda_{\tilde{f}}(n) \log\left(\frac{x}{n}\right) 
&\ll&
\sqrt{k\log k}~.~N^{\frac{6}{5} + \frac{\Delta}{2}} 
~.~e^{(a_1 ~\frac{\log(N+1)}{\log\log(N+2)})}~.
~~x^{\frac{1}{5}}, \\
\sum_{n \leq x} \lambda_{\tilde{f}}(n)
&\ll&  
\sqrt{k \log k}~.~N^{\frac{6}{5} + \frac{\Delta}{2}} 
~.~e^{(a_2 ~\frac{\log(N+1)}{\log\log(N+2)})}~.
~~x^{\frac{3}{5} + \frac{\Delta}{2}}, \nonumber
\end{eqnarray}
where $\Delta = \frac{1}{100}$ and
$a_1, a_2 >0$ are absolute constants. 
\end{prop}

We need the following Lemmas to prove Proposition \ref{thle}.
\begin{lem}\label{hecke-le}
For a square-free $N$, let $f$ be a normalized
newform of weight $k$ for $\Gamma_0(N)$. 
Then for any $ - \Delta < \sigma < 1 + \Delta$ with
$\Delta = \frac{1}{100}$ and any $t \in \R$,
we have 
$$
| L(f, \sigma + it)| 
~~\ll~~
N^{\frac{1 + \Delta - \sigma}{2}} ~.~
(3 + |t|)^{ 1 + \Delta - \sigma},
$$
where the constant in $\ll$ is absolute. 
\end{lem}

\medskip
\noindent
{\bf Proof of Lemma \ref{hecke-le}.}
It is easy to see that
$$
| L(f, 1+ \Delta + it)| ~~\ll~~ 1,
\phantom{m}
\forall  \in \R,
$$
where $L(f, s)$ is the Hecke $L$-function attached
to $f$. Now using the functional equation
$$
L^*(f, s) = L^*(f , 1-s),
$$
where 
$L^*(f, s) = N^{s/2} (2\pi)^{-s} \Gamma(s) L(f, s)$,
we get
$$
| L(f, - \Delta + it)| ~~\ll~~ N^{1/2 + \Delta} ~.~ |1 + it|^{1 + 2\Delta}
$$
for any $t \in \R$.
Finally using Proposition \ref{rad}, we get
for any $- \Delta < \sigma < 1+ \Delta$ and 
any $t \in \R$ that
$$
| L(f, \sigma + it)| ~~\ll~~ 
N^{\frac{(1 + \Delta - \sigma)}{2}} 
~.~(3 + |t|)^{1+ \Delta - \sigma},
$$ 
where the constant in $\ll$ is absolute.
This completes the proof of Lemma \ref{hecke-le}.

\bigskip
\begin{lem}\label{hecke-le1}
For square-free integer $N$ and $d | N$, let
$f \in S_k^{\text{new}}(d)$ be normalized Hecke 
eigen forms.  For $\epsilon_{f, p} \in \{ \pm 1\}$, define
$$
F:= f ~|~ \prod_{ p | \frac{N}{d}} ( 1 + \epsilon_{f,p} p^{k/2} V_p)
~=~
\sum_{ \delta | \frac{N}{d}} 
\epsilon_{ f, \delta}\delta^{k/2} f |V_{\delta},
\phantom{m}
\epsilon_{f, \delta} := \prod_{ p | \delta} \epsilon_{f, p}
$$
as in Lemma \ref{sp-basis}. Also let  $\epsilon > 0$. 
Then for $- \Delta < \sigma < 1 + \Delta$ with
$\Delta = \frac{1}{100}$ and any $t \in \R$, 
we have
\begin{eqnarray*}
\left|  
L(F, \sigma + it ) 
\right|
~~\ll~~
\tau_{1/2 - \sigma}(N) ~.~ 
N^{\frac{(1 + \Delta - \sigma)}{2}}
~.~(3 + |t|)^{1+ \Delta - \sigma},
\end{eqnarray*}
where the constant in $\ll$ is absolute and 
$\tau_{\ell}(N) := \sum_{d|N} d^{\ell}$. 
\end{lem}

\bigskip
\noindent
{\bf Proof of Lemma \ref{hecke-le1}.}
By definition, for any $z \in \mathcal{H}$, we have
 $$
F(z)
~=~ 
 \sum_{ \delta | \frac{N}{d}} 
\epsilon_{ f, \delta}~\delta^{k/2}( f |V_{\delta})(z)
~=~
\sum_{ \delta | \frac{N}{d}} 
\epsilon_{ f, \delta}~\delta^{k/2} f(\delta z).
$$
This implies that
$$
\lambda_F(n) ~=~ \sum_{ \delta | \frac{N}{d}} 
\epsilon_{ f, \delta}~\sqrt{\delta} ~\lambda_f(\frac{n}{\delta}).
$$
Hence for any $s \in \C$ with $\Re(s) > 1$,
we have
$$
L(F, s) ~=~ ( \sum_{ \delta | \frac{N}{d}} 
\epsilon_{ f, \delta}~\delta^{-s + 1/2} ) ~L(f, s).
$$
Thus for any $0 < \sigma < 1+ \Delta$ and any $t \in \R$,
we have
$$
| L( F, \sigma + it) | 
~~\ll~~
\tau_{1/2 - \sigma}(N) ~.~ |L(f, \sigma + it)|
~~\ll~~
\tau_{1/2 - \sigma}(N)
~.~N^{\frac{1 + \Delta - \sigma}{2}} ~.~
(3 + |t|)^{ 1 + \Delta - \sigma}.
$$
This completes the proof of Lemma \ref{hecke-le1}.

\medskip
We now complete the proof of Proposition \ref{thle}.
\begin{proof}
As before, we write
$$
\tilde{f} := \sum_{\tau=1}^{d_{k,N}} \alpha_{\tau} F_{\tau},
$$
where $d_{k,N} := \text{ dim }S_k(N)$ and
$\{ F_{\tau} \}_{1\le \tau \le d_{k, N}}$ is the special
orthogonal basis of $S_k(N)$ (in some fixed order) 
mentioned in Lemma \ref{sp-basis} and set
$$
A_{\tilde{f}} :=  \sum_{\tau =1}^{d_{k,N}} |\alpha_{\tau}|.
$$
Since
$$
L(\tilde{f}, s) 
~=~ 
\sum_{\tau} \alpha_{\tau} ~L(F_{\tau}, s),
$$
using Lemma \ref{hecke-le1}, we get
for any $- \Delta < \sigma < 1 + \Delta$
and any $t \in \R$
\begin{eqnarray}\label{he-re}
\left|  
L(\tilde{f},  \sigma + it ) 
\right|
&\ll&
A_{\tilde f}
~.~
\tau_{1/2 - \sigma}(N)
~.~N^{\frac{1 + \Delta - \sigma}{2}} ~.~
(3 + |t|)^{ 1 + \Delta - \sigma} \nonumber \\
&\ll&
\sqrt{k} ~.~\sqrt{\log k}
~.~ e^{(b_1 ~\sqrt{\frac{\log(N+1)}{\log\log(N+2)})}}
~.~ \tau_{1/2 - \sigma}(N)
~.~N^{1 + \frac{\Delta - \sigma}{2}} ~.~
(3 + |t|)^{ 1 + \Delta - \sigma}, 
\end{eqnarray}
where $b_1$ is an absolute constant
and also we have used the estimate
$$
A_{\tilde f} 
~~\ll~~
\sqrt{k N \log k}
~.~ e^{(b_1 ~\sqrt{\frac{\log(N+1)}{\log\log(N+2)})}}.
$$
In particular, for $\sigma = \frac{1}{5}$, we have
\begin{equation}\label{he-re1}
\left|  
L(\tilde{f},  ~\frac{1}{5} + it ) 
\right|
~~\ll~~
\sqrt{k \log k} 
~.~ N^{\frac{6}{5} + \frac{\Delta}{2}}
~.~
e^{(a_1 ~\frac{\log(N+1)}{\log\log(N+2)})}
~.~
(3 + |t|)^{ \frac{4}{5} + \Delta} 
\phantom{m} \forall ~ t \in \R,
\end{equation}
where $a_1$ is an absolute constant. 
Then
$$
\sum_{ n \le x} \lambda_{\tilde{f}}(n) \log \left( \frac{x}{n} \right)
~~=~~
\frac{1}{2\pi i}
 \int_{ 1+ \frac{\Delta}{2} - i\infty}^{1
 + \frac{\Delta}{2} + i\infty} L(\tilde{f}, s)~ \frac{x^s}{s^2} ~ds.
$$
Now by shifting the line of integration to the line 
$\sigma = \frac{1}{5}$, we get
$$
 \sum_{ n \le x} \lambda_{\tilde{f}}(n) \log \left( \frac{x}{n} \right)
 ~~\ll~~
\sqrt{k\log k}
~.~ N^{\frac{6}{5} + \frac{\Delta}{2}}
~.~
e^{(a_1 ~\frac{\log(N+1)}{\log\log(N+2)})}
~.~ x^{\frac{1}{5}}.
$$
For the second part of the Proposition, we proceed as 
in Proposition \ref{secle}. First
note that
\begin{eqnarray}\label{last}
|\lambda_{\tilde{f}}(n) |
~~=~~
|\sum_{\tau} \alpha_{\tau} \lambda_{F_{\tau}}(n)| 
&\ll&
A_{\tilde f} ~.~ \tau(N)~.~ \sqrt{N}~.~ \tau(n) \\
&\ll&
\sqrt{k \log k}~.~ N
.~ e^{(a_1' ~\frac{\log(N+1)}{\log\log(N+2)})} .~ \tau(n), \nonumber
\end{eqnarray}
where $a_1'$ is an absolute constant
and hence by Perron's formula for $x\not\in \Z$, we have
$$
\sum_{ n \le x} \lambda_{\tilde{f}}(n)
~~=~~
\frac{1}{2\pi i}
 \int_{ 1+ \frac{\Delta}{2} - iT}^{1
 + \frac{\Delta}{2} + iT} L(\tilde{f}, s)~ \frac{x^s}{s} ~ds
 ~~+~~
 O\left(\sqrt{k \log k}. ~N. ~e^{(a_1' ~\frac{\log(N+1)}{\log\log(N+2)})} 
 .~~\frac{x^{1+ \Delta}}{T} \right),
$$
where $1 \le T \le x$ to be chosen later. Then shifting the line of
integration, we get
\begin{eqnarray*}
\frac{1}{2 \pi i} \int_{1 + \frac{\Delta}{2} - iT}^{1 + \frac{\Delta}{2} + i T} 
L(\tilde{f}, s)~ \frac{x^s}{s^2} ~ds
&=&
\frac{1}{2\pi i} ( \int_{ 1 +  \frac{\Delta}{2}  - iT}^{\frac{1}{5}  - iT} 
 ~+~   \int_{\frac{1}{5} - iT}^{\frac{1}{5} + iT}
~+~ \int_{\frac{1}{5}  + iT}^{1 +  \frac{\Delta}{2} + iT} )
~L(\tilde{f}, s) ~\frac{x^s}{s} ~ds\\
&=&
Y_1 ~+~ Y_2 ~+~ Y_3, \phantom{m}\text{say}.
\end{eqnarray*}
Using equation \eqref{he-re1}, we get
\begin{eqnarray}\label{dui}
|Y_2| 
&\ll&
\sqrt{k \log k}~.~ N^{\frac{6}{5} + \frac{\Delta}{2}}.~
e^{(a_1 ~\frac{\log(N+1)}{\log\log(N+2)})}
~.~ x^{\frac{1}{5}}.~ T^{4/5 + \Delta}~.
\end{eqnarray}
Now using \eqref{he-re}, we get
\begin{eqnarray*}
|Y_1 ~+~ Y_3 | 
&\ll&
\sqrt{k \log k} ~.~ e^{(a_2 ~\frac{\log(N+1)}{\log\log(N+2)})}
 .~N^{\frac{6}{5} + \frac{\Delta}{2}} .~
 \frac{1}{T}
 \left\{ \text{max}_{ \frac{1}{5} \le \sigma \le 1 + \frac{\Delta}{2}}
\left( \frac{x}{T}  \right)^{\sigma} \right\}~.~ T^{1+ \Delta} , \\
&\ll&
\sqrt{k \log k} ~.~ e^{(a_2 ~\frac{\log(N+1)}{\log\log(N+2)})}
 .~N^{\frac{6}{5} + \frac{\Delta}{2}} 
 ~.~~ \frac{x^{1 + \frac{\Delta}{2}}}{T^{1 - \frac{\Delta}{2}}}
\end{eqnarray*}
where the constant in $\ll$ and the constant $a_2 >0$ are 
absolute. Now by choosing $T = x^{1/2}$,
we get
\begin{eqnarray*}
\sum_{n \leq x} \lambda_{\tilde{f}}(n) 
&\ll&
 \sqrt{k \log k}~.~ e^{(a_2 ~\frac{\log(N+1)}{\log\log(N+2)})}
 .~N^{\frac{6}{5} + \frac{\Delta}{2}} ~.
~~x^{ \frac{3}{5} + \frac{\Delta}{2}}.
\end{eqnarray*}
This completes the proof. 
\end{proof}

We are now in a position to complete the proof
of Theorem \ref{thmthree}.

\bigskip
\noindent
{\bf Proof of Theorem \ref{thmthree}.} 
Let $D(k, N)$ be as in Proposition \ref{firstle} and
as before we normalise $f$ by the following condition
\begin{equation}\label{nor}
D(k, N) . < f, f > =1,
\end{equation}
where $ < ,>$ is the Petersson inner product.
This means we replace $f$ by 
$$
\tilde{f} := \frac{f}{ \sqrt{D(k,N)} || f ||}.
$$
Let $\tilde{\beta}(n)$ be the Fourier coefficients
of $\tilde{f}$ and $\tilde{\beta}(n) := \lambda_{\tilde{f}}(n) n^{\frac{k-1}{2}}$.
Note that $\tilde{\beta}(n)$ and $\lambda_{\tilde{f}}(n)$ have the same sign.
Hence it is sufficient to show the sign change
of $\lambda_{\tilde{f}}(n)$ in the desired range.

Assume the contrary, i.e. $\lambda_{\tilde{f}}(n)$ has constant
sign for $n \in (x, x+h]$, where $h= x^{13/14 + \epsilon}$
and $\epsilon >0$. Without loss of generality,
we can assume that $0 < \epsilon <1/200$ as constant sign
in a bigger interval implies constant sign in a smaller interval. 
Since we can always replace $\tilde{f}$
by $- \tilde{f}$, we can assume that $\lambda_{\tilde{f}}(n) \ge 0$ for
all $n \in (x, x+h]$. 

By Cauchy-Schwartz inequality, we have
\begin{eqnarray}\label{imp}
&&
\sum_{ x < n \leq x+h} \lambda_{\tilde{f}}^2(n) \log^2\left(\frac{x + h}{n}\right) \\
&\leq& 
\left(\sum_{ x < n \leq x+h} |\lambda_{\tilde{f}}(n)| \log^2
\left(\frac{x + h}{n}\right) \right)^{1/2} 
\left(\sum_{ x < n \leq x+h} |\lambda_{\tilde{f}}(n)|^3 
\log^2\left(\frac{x + h}{n}\right)\right)^{1/2}  \nonumber \\
&=& 
\left(\sum_{ x < n \leq x+h} \lambda_{\tilde{f}}(n) \log^2
\left(\frac{x + h}{n}\right)\right)^{1/2}
\left(\sum_{ x < n \leq x+h} \lambda_{\tilde{f}}(n)^3 \log^2
\left(\frac{x + h}{n}\right)\right)^{1/2}, 
\nonumber
\end{eqnarray}
by our assumption on $\lambda_{\tilde{f}}(n)$.
We will estimate the left hand side of \eqref{imp} from below
and the two terms on the right hand side from
above and derive a contradiction for sufficiently large $x$.
Note that
\begin{eqnarray*}
&&
\sum_{x < n \leq x+h} \lambda_{\tilde{f}}^2(n) \log^2\left(\frac{x+h}{n}\right) \\
&=&
\sum_{ n \leq x+h} \lambda_{\tilde{f}}^2(n) \log^2\left(\frac{x + h}{n}\right)
~-~
\sum_{ n \leq x} \lambda_{\tilde{f}}^2(n) \log^2\left(\frac{x + h}{n}\right) \\
&=&
\sum_{ n \leq x+h} \lambda_{\tilde{f}}^2(n) \log^2\left(\frac{x + h}{n}\right)
~-~
\sum_{ n \leq x} \lambda_{\tilde{f}}^2(n) \log^2\left(\frac{x}{n}\right) 
~-~
\log^2\left(1+ \frac{h}{x}\right)\sum_{n \leq x} \lambda_{\tilde{f}}^2(n)  \\
&& \phantom{mmmmm}
~-~~~
2\log\left(1 + \frac{h}{x}\right)\sum_{n \leq x} \lambda_{\tilde{f}}^2(n)
 \log\left(\frac{x}{n}\right).
\end{eqnarray*}
We know from page 538 of \cite{CK} that
$$
\sum_{n \leq y} \lambda_{\tilde{f}}^2(n) \log^2\left(\frac{y}{n}\right) 
~=~
\frac{12}{\pi^2} \prod_{ p |N \atop \text{ prime }} (1 + 1/p)^{-1} ~y
~+~  O \left( k^{3/2} .~N^2. ~\log^5k.~\sqrt{\Phi_k(N)}.~
e^{\bar{c}_1 \frac{\log (N+1)}{\log\log (N+2)}}. ~y^{1/2} \right),
$$
where $\Phi_k(N)$ is as in \eqref{new} and
$\bar{c}_1 > 0$ is an absolute constant.
Hence we have
\begin{eqnarray}\label{eq-r}
&& \sum_{ n \leq x+h} \lambda_{\tilde{f}}^2(n) \log^2\left(\frac{x + h}{n}\right)
~~-~~ 
\sum_{ n \leq x} \lambda_{\tilde{f}}^2(n) \log^2\left(\frac{x}{n}\right) \\
&&
~~=~~ 
\frac{12}{\pi^2} \prod_{ p |N \atop \text{ prime }} (1 + 1/p)^{-1} ~h
~+~  O\left( k^{3/2}. ~N^2. ~\log^5k.~\sqrt{\Phi_k(N)}~.~
e^{\bar{c}_1 \frac{\log (N+1)}{\log\log (N+2)}} ~x^{1/2} \right). \nonumber
\end{eqnarray}
By Proposition \ref{firstle} and Proposition \ref{secle}, we know that
\begin{eqnarray*}
\sum_{n \leq x} \lambda_{\tilde{f}}^2(n) \log\left(\frac{x}{n}\right) 
&=&
\frac{6}{\pi^2} \prod_{ p |N \atop \text{ prime }} (1 + \frac{1}{p})^{-1} ~x
~+~  
O \left( (kN)^{\frac{5}{4} + \Delta}
~.~N^{\frac{\Delta(1 + 4\Delta )}{4(1+ 2\Delta)}}.~\sqrt{N}\right. \\
&&
\phantom{mmmmmmmmmmmmmmmmmm}
\left.  
.~e^{(c_3 ~\frac{\log(N+1)}{\log\log(N+2)})}
~. ~\log k~.~x^{3/4} \right) \\
\sum_{n \leq x} \lambda_{\tilde{f}}^2(n)  
&=&
\frac{6}{\pi^2} \prod_{ p |N \atop \text{ prime }} (1 + \frac{1}{p})^{-1} ~x 
~+~
\phantom{mm}  
O_{\epsilon} \left( (kN)^{\frac{3}{2} + \frac{\Delta}{2}}
~.~N^{\frac{\Delta(1 + \Delta)}{2(1+ 2\Delta)}} 
~.~\sqrt{N}~.~\log k \right. \\
&&
\phantom{mmmmmmmmmmmmmmmmmm}
\left. .~e^{(c_4 \frac{\log(N+1)}{\log\log(N+2)})}.
~x^{\frac{3}{4}(1 + \Delta) + \epsilon}
\right),
\end{eqnarray*}
where $c_3 >0, c_4 >0$ are absolute constants
and $\Delta = \frac{1}{100}$.
Thus by using the above identities, we get
\begin{eqnarray*}
&&
\sum_{ x < n \leq x+h} \lambda_{\tilde{f}}^2(n) \log^2\left(\frac{x + h}{n}\right) \\
&=&
 \frac{12}{\pi^2} \prod_{ p |N \atop \text{ prime }} (1 + \frac{1}{p})^{-1} ~h
~-~
\log^2\left(1+ \frac{h}{x}\right)\sum_{n \leq x} \lambda_{\tilde{f}}^2(n)  
~-~~~
2\log\left(1 + \frac{h}{x}\right)\sum_{n \leq x} \lambda_{\tilde{f}}^2(n)
 \log\left(\frac{x}{n}\right) \\
 &&
 \phantom{mmm}
 ~+~~
 \phantom{m}
 O \left( k^{3/2} .~N^2 .~\log^5k.~\sqrt{\Phi_k(N)}.
 ~e^{\bar{c}_1 \frac{\log (N+1)}{\log\log (N+2)}} 
 .~x^{1/2} \right)\\
&=& 
\frac{2}{\pi^2} \prod_{ p |N \atop \text{ prime }} 
(1 + \frac{1}{p})^{-1} ~\frac{h^3}{x^2}
~+~
O_{\epsilon} \left( (kN)^{\frac{3}{2} + \frac{\Delta}{2}}~
.~N^{\frac{\Delta(1 + \Delta)}{2(1+ 2\Delta)}} 
.~\sqrt{N \Phi_k(N)}.~\log^5k.
~ e^{(\tilde{c}_5 \frac{\log(N+1)}{\log\log(N+2)})}~.
~x^{5/7 + 4\epsilon}   \right),
\end{eqnarray*}
where the last equality follows from the identity
$$
\log ( 1 + \frac{h}{x}) ~=~ \frac{h}{x} ~-~ \frac{h^2}{2x^2} 
~+~ \frac{h^3}{3x^3} ~+~ O(\frac{h^4}{x^4})
$$
for $x \gg 1$ and $\tilde{c}_5 >0$ is an absolute constant. 
Thus we have
\begin{equation}\label{RS}
\sum_{ x < n \leq x+h} \lambda_{\tilde{f}}^2(n) 
\log^2\left(\frac{x + h}{n}\right)
~~ \gg~ \frac{h^3}{x^2 \log\log (N+2)}~,
\end{equation}
where the constant in $\gg$ is absolute and 
$$
x ~\gg_{\epsilon}~ \left( (kN)^{\frac{3}{2} + \frac{\Delta}{2}}~
.~N^{\frac{\Delta(1 + \Delta)}{2(1+ 2\Delta)}} 
.~\sqrt{N \Phi_k(N)}.~\log^5k.
~ e^{(c_5 \frac{\log(N+1)}{\log\log(N+2)})}\right)^{\frac{14}{1-14\epsilon}},
$$
where $c_5 > 0$ is an absolute constant.
From \eqref{last}, we get 
\begin{eqnarray*}
\sum_{ x < n \leq x+h} \lambda_{\tilde{f}}(n)^3 \log^2\left(\frac{x + h}{n}\right)
\ll A(k, N).  ~x \log^9 x,
\end{eqnarray*}
where 
$$
A(k, N) :=  k^{3/2}.~N^3.~ \log^{5/2}k.~ e^{c_6' ~\frac{\log(N+1)}{\log\log(N+2)}}
$$ 
and the constant $c_6' >0$ and the constant in $\ll$ are absolute. 
Now from page 540, equation (8.7) of \cite{CK} and using
Proposition \ref{thle}, we get
\begin{eqnarray*}
|\sum_{ x < n \le x+h} \lambda_{\tilde{f}}(n) \log^2\left(\frac{x +h}{n}\right) |
&=& 
|\sum_{ n \le x+h} \lambda_{\tilde{f}}(n) \log^2\left(\frac{x + h}{n}\right)
~-~ \sum_{ n \le x} \lambda_{\tilde{f}}(n) \log^2\left(\frac{x}{n}(1 + \frac{h}{x})\right)| \\
&\ll&  
 A(k, N).~ x^{1/2} ~+~ 
 | 2 \log(1+ \frac{h}{x})~ \sum_{ n \le x} \lambda_{\tilde{f}}(n) \log\left(\frac{x}{n}\right)|  \\
&& 
\phantom{mm}
~+~
|\log^2(1 + \frac{h}{x}) ~\sum_{ n \le x} \lambda_{\tilde{f}}(n)|  \\
&\ll&  
k^{3/2}.~N^3.~ \log^{5/2}k.
~ e^{\tilde{c}_6 ~\frac{\log(N+1)}{\log\log(N+2)}}.~ x^{1/2} ,
\end{eqnarray*}
where the constant $\tilde{c}_6 > 0$ and 
the constant in $\ll$ are absolute.
Therefore an upper bound for the right hand side of
the inequality \eqref{imp} is 
$\ll_{\epsilon} x^{3/4 + \epsilon}$
which contradicts the lower bound in \eqref{RS}
when 
\begin{eqnarray*}
x 
&\gg_{\epsilon}& 
\text{  max }\left\{ \left[k^{3/2}.~N^3.~ 
\log^{5/2}k.~ e^{c_6 ~\frac{\log(N+1)}{\log\log(N+2)}}
\right]^{\frac{28}{1+ 56\epsilon}},  \right. \\
&&
\phantom{mmm}
\left.
~~\left[ (kN)^{\frac{3}{2} + \frac{\Delta}{2}}
~.~N^{\frac{\Delta(1 + \Delta)}{2(1+ 2\Delta)}} 
.~\sqrt{N \Phi_k(N)}.~\log^5k.
~ e^{(c_5 \frac{\log(N+1)}{\log\log(N+2)})}~\right]^{\frac{14}{1-14\epsilon}}
\right\},
\end{eqnarray*}
where $c_5, c_6> 0$ are absolute constants.
Note that for any sufficiently small $\eta$, one knows
that $\log \Phi_k(N)  \le (1.7 + \eta) \log N$ if $N \ge k \ge 16$
and  $\log \Phi_k(N)  \le (1.7 + \eta) \log k$
if $k \ge N \ge 16$. Hence for any $0 < \epsilon < 1/200$,
one has
\begin{eqnarray*}
&\text{  max }&\left\{ \left[k^{3/2}.~N^3.~ 
\log^{5/2}k.~ e^{c_6 ~\frac{\log(N+1)}{\log\log(N+2)}}
\right]^{\frac{28}{1+ 56\epsilon}},  \right.  \phantom{mmmm}\\
&&
\phantom{mmm}
\left.
~~\left[ (kN)^{\frac{3}{2} + \frac{\Delta}{2}}
~.~N^{\frac{\Delta(1 + \Delta)}{2(1+ 2\Delta)}} 
.~\sqrt{N \Phi_k(N)}.~\log^5k.
~ e^{(c_5 \frac{\log(N+1)}{\log\log(N+2)})}~\right]^{\frac{14}{1-14\epsilon}}
\right\} \\
\phantom{mm}
&\ll&
k^{42}.~N^{84}.~\log^{80}k.~e^{(c_7 \frac{\log(N+1)}{\log\log(N+2)})},
\end{eqnarray*}
where $c_7 > 0$ is an absolute constant.
Thus there exists $n_1 \in (x, x+h]$ such that
$\lambda_{\tilde{f}}(n_1) < 0$ and therefore
$\beta(n_1) <0$. This concludes the
proof of the theorem.

\smallskip

\section{\large Proof of Theorem \ref{thmtwo}}

\smallskip

Let the notations be as in Theorem \ref{thmone}, $\epsilon > 0$
and
\begin{eqnarray*}
x 
&\gg_{\epsilon}& 
k^{42}.~N^{84}.~\log^{80}k.~e^{(c_7 \frac{\log(N+1)}{\log\log(N+2)})},
\end{eqnarray*}
where $c_7 > 0$ is an absolute constant.
Also let
$\epsilon_1 := \epsilon/4$ and $h_1 := x^{13/14 + \epsilon_1}$.
Using Theorem \ref{thmthree} on $\chi_{\alpha}$, 
we get $n_1, n_2 \in (x, x + h_1]$ such that
$B(n_1) > 0$ and $B(n_2) < 0$. Then
arguing as in the proof of Theorem \ref{thmone},
there exist a pair $(n, r)$ such that
$c(n, r) < 0$ where $n= n_1$ or $n=n_2$ and $r=r_1$
or $r= r_2$.
Since
$$
c(n, r) = a( \begin{pmatrix} n &  r /2  \\   r /2 &  m_0 \end{pmatrix}),
$$
we have 
$$
a( \begin{pmatrix} n &  r /2  \\   r /2 &  m_0 \end{pmatrix}) < 0.
$$
Now 
$$
\text{ tr } (\begin{pmatrix} n &  r /2  \\   r /2 &  m_0 \end{pmatrix})
~=~  n + m_0
$$
with $x < n + m_0 \le x + h_1 + c(k, N)$,
where 
$$
c(k, N) := \frac{4}{3\sqrt{3} \pi} kN^3 
\prod_{p|N \atop p \text{ prime }}(1 + 1/p)(1+ 1/p^2).
$$ 
Hence 
$$
n + m_0  \le x + 2h _1 \le x + h.
$$ 
This completes the proof of Theorem \ref{thmtwo}.

\bigskip
\noindent
{\bf Acknowledgements.} We would like to thank S. B\"oecherer
and W. Kohnen for providing us certain references. We also
thank C. Poor for providing us the bound in equation \eqref{bound}.
The authors would like to thank the referee for his/her
suggestions which improved the exposition of the paper and also
for pointing out an inaccuracy in the previous version of the paper.
We would also like to thank R. Balasubramanian, Biplab Paul
and Purusottam Rath for going through
an earlier version of the paper. 
Part of the work was done when the first author was 
visiting ICTP as a regular associate and would
like to thank ICTP for the excellent working facilities.


\begin{thebibliography}{100}

\bibitem{CK}
Y. Choie and W. Kohnen, 
{\em The first sign change of Fourier 
coefficients of cusp forms},
Amer. J. Math. {\bf 131} (2009), no. 2, 517--543.


\bibitem{CGK}
Y-J Choie, S. Gun and W. Kohnen, 
{\em An explicit bound for the first sign 
change of the Fourier coefficients},
Int. Math. Res. Not. IMRN {\bf 12} (2015), 
3782--3792. 


\bibitem{EZ}
M. Eichler and D. Zagier, 
The theory of Jacobi forms, 
Progress in Math. {\bf 55},
{\em Birkh\"auser}, 1985.


\bibitem{GS}
A. Ghosh and P. Sarnak,
{\em Real zeros of holomorphic Hecke cusp forms},
J. Eur. Math. Soc. 14 (2012), 465--487.


\bibitem{IKS}
H. Iwaniec, W. Kohnen and J. Sengupta,
{\em The first negative Hecke eigenvalue},
Int. J. Number Theory {\bf 3} (2007), 355--363.


\bibitem{IK}
H. Iwaniec and E. Kowalski, {\em Analytic Number Theory},
American Mathematical Society, Providence, 
RI, 2004.


\bibitem{SJ}
S. Jesgarz,
{\em Vorzeichenwechsel von Fourierkoeffizienten 
von Siegelschen Spitzenformen},
Diploma Thesis (unpublished), 
Univ. of Heidelberg 2008.


\bibitem{KKP}
M. Knopp, W. Kohnen and W. Pribitkin,
{\em On the signs of Fourier coefficients of
cusp forms}, Ramanujan J. {\bf 7} (2003), no. 1-3, 
269--277.


\bibitem{WK}
W. Kohnen, {\em Sign changes of Hecke 
eigenvalues of Siegel
cusp forms of genus two}, 
Proc. Amer. Math. Soc. {\bf 135} (2007), 
997--999.


\bibitem{KS}
W. Kohnen and J. Sengupta,
{\em On the first sign change of Hecke 
eigenvalues of newforms},
Math. Z. {\bf 254} (2006), no. 1, 173--184.

\bibitem{KS1}
W.  Kohnen and J. Sengupta,  {\em The first negative Hecke 
eigenvalue of a  Siegel cusp form of genus two},  Acta Arith. {\bf 129} 
(2007), no. 1, 53--62.
 
 
\bibitem{KLS}
W. Kohnen,  Y. Lau and I. Shparlinski, 
{\em On the number of sign changes of 
Hecke eigenvalues of newforms},
J. Aust. Math. Soc. {\bf 85} (2008), 
no. 1, 87--94. 


\bibitem{WL}
W-C Winnie Li, {\em L-series of Rankin type and their 
functional equations}, Math. Ann. {\bf 244} (1979), 
no. 2, 135--166.


\bibitem{AO}
A. P. Ogg, {\em On a convolution of L-series},
Invent. Math. {\bf 7} (1969), 297--312. 


\bibitem{KM}
 K. Matom\"aki,
{\em On signs of Fourier coefficients of cusp forms}, 
Math. Proc. Camb. Phil. Soc. {\bf 152} (2012), 
2007--2022.


\bibitem{RM1}
M.R. Murty, 
{\em Oscillations of Fourier coefficients of 
modular forms},  Math. Annalen {\bf 262} (1983), no 4
431--446. 


\bibitem{RM}
M. Ram Murty,  
Problems in analytic number theory,
Graduate Texts in Mathematics {\bf 206}, 
{\em Springer-Verlag}, New York, 2001.


\bibitem{RR}
R. A. Rankin, 
{\em Contributions to the theory of Ramanujan's 
function $\tau(n)$ and similar arithmetical functions}, 
Proc. Cambridge Philos. Soc. {\bf 35} (1939) 351--372.


\bibitem{PY}
C. Poor and D.S. Yuen,
{\em Dimensions of cusp forms for 
$\Gamma_0(p)$ in degree two
and small weights}, Abh. Math. Sem. Univ. 
Hamburg {\bf 77}
 (2007), 59--80.


\bibitem{PY1}
C. Poor and D. Yuen, 
{\em Paramodular cusp forms} 
Math. Comp. {\bf 84} (2015), no. 293, 
1401--1438. 


\bibitem{HR}
H. Rademacher, {\em On the Phragm\'en-Lindel\"of theorem and some 
applications}, Math. Z ~{\bf 72} (1959), 192--204.


\bibitem{GR}
G. Robin, {\em Estimation de la fonction de Tchebychef $\theta$ 
sur le k-i\'eme nombre premier et grandes valeurs de la 
fonction $\omega(n)$ nombre de diviseurs 
premiers de n},  Acta Arith. {\bf 42} (1983), 367--389. 


\bibitem{GT}
G. Tenenbaum, {\em Introduction to analytic and probabilistic 
number theory},
Cambridge Studies in Advanced Mathematics 46,
Cambridge University Press, Cambridge, 1995.


\end{thebibliography}
\end{document}